\documentclass{article}
\usepackage{amssymb}
\usepackage{amsmath,amsfonts}
\usepackage{amsmath, latexsym, amsfonts, amssymb, amsthm, amscd, graphics, xcolor,graphicx}
\usepackage{amsmath, latexsym, amsfonts, amssymb, amsthm, amscd}

\newcommand{\E}{\mathbb{E}}
\newcommand{\Tr}{{\rm Tr}}
\newcommand{\C}{\mathbb{C}}
\newcommand{\R}{\mathbb{R}}
\newcommand{\1}{1\!\!{\sf I}}
\newcommand{\vers}{\mathop{\longrightarrow }}
\newtheorem {theoreme} {{\bf Theorem}} [section]

 \newtheorem{theorem}[theoreme]{Theorem}
 \newtheorem{lemme}[theoreme]{Lemma}
\newtheorem{remark}[theoreme]{Remark}
\newtheorem{corollaire}[theoreme]{Corollary}
\newtheorem{proposition}[theoreme]{Proposition}
\title{Limiting eigenvectors of outliers for Spiked  Information-Plus-Noise type matrices
}

\author{M. Capitaine\thanks{CNRS, Institut de Math\'ematiques de Toulouse, 
  F-31062 Toulouse Cedex 09. 
E-mail: mireille.capitaine@math.univ-toulouse.fr }}

\date{}
\begin{document}
\maketitle
\begin{abstract}
We consider an Information-Plus-Noise type matrix where the Information matrix is a spiked  matrix.
When some eigenvalues of the random matrix separate from the bulk, we study how the corresponding eigenvectors project
onto those of the spikes. Note that, in an Appendix, we present alternative versions of the earlier results of \cite{BaiSilver} (``noeigenvalue outside the support of the deterministic equivalent measure'') and \cite{MC2014} (``exact separation phenomenon'') where we remove some   technical assumptions that were difficult to handle.
\end{abstract}
\section{Introduction}\label{no}

\noindent In this paper, we  consider the so-called Information-Plus-Noise type model
\begin{equation}\nonumber 
M_N= \Sigma_N \Sigma_N^*
\mbox{~~where~~}
 \Sigma_N =
 \sigma \frac{X_N}{\sqrt{N}}+A_N, \end{equation}
\noindent defined as follows.
\begin{itemize}
\item $n=n(N)$, $n \leq N$, $c_N=n/ N\rightarrow_{N\rightarrow +\infty} c \in]0;1].$
\item $\sigma \in ]0;+\infty[.$
\item $X_N=[X_{ij}]_{1\leq i \leq n; 1\leq j \leq N}$ where 
 $\{X_{ij}, i\in \mathbb{N}, j  \in \mathbb{N}\}$ is  an infinite set of complex random variables such that $\{\Re (X_{ij}), \Im (X_{ij}), i\in \mathbb{N}, j  \in \mathbb{N}\}$ are independent   centered random variables   with variance $1/2$ and satisfy \begin{enumerate} \item There exists  $K>0$ and a random variable $Z$ with finite fourth moment for which there exists $x_0>0$ and  an  integer number $n_0>0$ such that, for any $x >x_0$ and any integer numbers $n_1,n_2 >n_0$, we have
\begin{equation}\label{condition}\frac{1}{n_1n_2} \sum_{i\leq n_1,j\leq n_2}P\left( \vert X_{ij}\vert >x\right) \leq KP\left(\vert Z \vert>x\right).\end{equation}
\item \begin{equation}\label{trois}\sup_{(i,j)\in \mathbb{N}^2}\mathbb{E}(\vert X_{ij}\vert^3)<+\infty. \end{equation}
\end{enumerate}
\item Let $\nu$ be a compactly supported probability measure on $\mathbb{R}$ whose support has a finite number of connected components. Let
$\Theta=\{\theta _1; \ldots ; \theta _J\}$ where  $\theta _1 > \ldots > \theta _J\geq 0$ are
$J$ fixed real numbers 
independent of $N$ which are outside the support of $\nu $.
Let $k_1,\ldots,k_J$ be fixed integer numbers independent of $N$ and  $r=\sum_{j=1}^J k_j$.
Let $\beta_j(N)\geq 0$, $r+1\leq j \leq n$, be such that
$\frac{1}{n} \sum_{j=r+1}^{n} \delta _{\beta _j(N)}$ weakly
converges to  $\nu $ and
 \begin{equation}\label{univconv}
 \max _{r+1\leq j\leq n} {\rm dist}(\beta _j(N),{\rm supp}(\nu ))\vers _{N \rightarrow \infty } 0\end{equation}
where ${\rm supp}(\nu )$ denotes the support of $\nu $. \\
Let $\alpha_j(N), j=1,\ldots,J$,  be real nonnegative  numbers such that 
$$\lim_{N \rightarrow +\infty} \alpha_j(N)=\theta_j.$$
Let $A_N$ be a  $n \times N $ deterministic matrix  such that, for each $j=1,\ldots,J$,  
   $\alpha_j(N) $ 
is an eigenvalue of $A_N A_N^*$ with  multiplicity $k_j$,  and the other eigenvalues of $A_N A_N^*$ are the 
$\beta_j(N)$, $r+1\leq j \leq n$. Note that  the empirical spectral measure of ${A_N A_N^*} $ weakly
converges to  $\nu $.
\end{itemize}

\begin{remark}
Note that assumption such as \eqref{condition} appears in \cite{CSBD}. It obviously holds if the $X_{ij}$'s
are identically distributed with finite fourth moment.
\end{remark}

 For any Hermitian $n\times n$ matrix $Y$, denote by $\rm{spect}(Y)$ its spectrum, by
$$\lambda_1(Y) \geq \ldots \geq \lambda_n(Y)$$
\noindent the ordered eigenvalues of $Y$ and by $\mu_{Y}$  the empirical spectral measure of $Y$: $$\mu _{Y} := \frac{1}{n} \sum_{i=1}^n \delta_{\lambda _{i}(Y)}.$$ For a probability measure $\tau $ on $\R$, 
 denote by $g_\tau $ its Stieltjes transform defined for $z \in \C\setminus \R$ by 
$$g_\tau (z) = \int_\R \frac{d\tau (x)}{z-x}.$$
When the $X_{ij}$'s are identically distributed, Dozier and Silverstein established 
in \cite{DozierSilver} that almost surely the empirical spectral measure $\mu _{M_N}$ of $M_N$ converges weakly  towards a nonrandom distribution $\mu_{\sigma,\nu,c}$  which is characterized in terms of its Stieljes transform which satisfies the following equation:
 for any $z \in \mathbb{C}^+$,
\begin{equation}\label{TS} g_{\mu_{\sigma,\nu,c}}(z)=\int \frac{1}{(1-\sigma^2cg_{ \mu_{\sigma,\nu,c}}(z))z- \frac{ t}{1- \sigma^2 cg_{ \mu_{\sigma,\nu,c}}(z)} -\sigma^2 (1-c)}d\nu(t).\end{equation}
This result of convergence was extended to independent but non identically
distributed random variables by Xie in \cite{Xi}. (Note that, in \cite{HLN} , the authors in-
vestigated the case where
$\sigma$
is replaced by a bounded sequence of real numbers.)
In \cite{MC2014}, the author carries on with the study of the  support of the limiting spectral measure previously investigated in \cite{DozierSilver2} and later
in  \cite{VLM, LV} and obtains  that there is a one-to-one relationship between  the complement of the limiting support  and some subset in the complement of the support of $\nu$ which is defined in (\ref{cale}) below.

\begin{proposition}\label{caractfinale}
Define differentiable functions   $\omega_{\sigma,\nu,c}$ and $\Phi_{\sigma,\nu,c}$  on respectively $ \mathbb{R}\setminus \mbox{supp}(\mu_{\sigma,\nu,c})$ and $ \mathbb{R}\setminus \mbox{supp}(\nu)$ by setting  \begin{equation}\label{defom}\omega_{\sigma,\nu,c} :\begin{array}{ll} \mathbb{R}\setminus \mbox{supp}(\mu_{\sigma,\nu,c}) \rightarrow \mathbb{R}\\
x \mapsto 
x (1- \sigma^2 c g_{ \mu_{\sigma,\nu,c}}(x))^2 -\sigma^2 (1-c)(1-\sigma^2 c g_{\mu_{\sigma,\nu,c}}(x))\end{array}\end{equation}
and  $$\Phi_{\sigma,\nu,c} :\begin{array}{ll} \mathbb{R}\setminus \mbox{supp}(\nu) \rightarrow \mathbb{R}\\
x \mapsto 
x (1+c \sigma^2g_{ \nu}(x))^2 + \sigma^2 (1-c) (1+ c \sigma^2 g_\nu(x))\end{array}.$$
Set \begin{equation}\label{cale} {\cal E}_{\sigma,\nu,c}:=\left\{ x \in  \mathbb{R}\setminus \mbox{supp}(\nu),  \Phi_{\sigma,\nu,c}^{'}(x) >0, g_\nu(x) >-\frac{1}{\sigma^2c}\right\}.\end{equation}
$\omega_{\sigma,\nu,c}$ is an increasing  analytic diffeomorphism with positive derivative  from $\mathbb{R}\setminus \mbox{supp}(\mu_{\sigma,\nu,c})$ to ${\cal E}_{\sigma,\nu,c}$, with inverse $\Phi_{\sigma,\nu,c}$.
\end{proposition}
Moreover,
 extending previous results in \cite{LV} and \cite{BGRao10} involving the  Gaussian case and finite rank perturbations,
 \cite{MC2014} establishes a one-to-one correspondance between the  $\theta_i$'s that belong to the set ${\cal E}_{\sigma,\nu,c}$ (counting multiplicity) and the outliers in the spectrum of $M_N$.
More precisely, setting 
\begin{equation}\label{defTheta}\Theta_{\sigma,\nu,c} = \left\{\theta \in \Theta,   \Phi_{\sigma,\nu,c}^{'}(\theta) >0, g_\nu(\theta) >-\frac{1}{\sigma^2c}\right\},\end{equation} and 
\begin{equation}\label{defSsigma}{\cal S}=\mbox{~supp~} (\mu_{\sigma,\nu,c}) \cup \left\{ \Phi_{\sigma,\nu,c}({\theta}), \theta \in \Theta_{\sigma,\nu,c} \right\},\end{equation}
we have the following results.
\begin{theorem}\label{inclusionth}\cite{MC2014}
 For any  $\epsilon>0$, 
$$\mathbb P[\mbox{for all large N}, \rm{spect}(M_N) \subset \{x \in \R , \mbox{dist}(x,{\cal S})\leq \epsilon \}]=1.$$
\end{theorem}
\begin{theorem}{\label{ThmASCV}}\cite{MC2014}
Let  $\theta_j$ be  in $ \Theta_{\sigma,\nu,c}$ 
and  denote by $n_{j-1}+1, \ldots , n_{j-1}+k_j$ the descending ranks of $\alpha_j(N)$ among the eigenvalues of $A_NA_N^*$.
Then the $k_j$ eigenvalues $(\lambda_{n_{j-1}+i}(M_N), \, 1 \leq i \leq k_j)$ 
converge almost surely outside the support of $\mu_{\sigma,\nu,c}$
towards $\rho _{\theta _j}:=\Phi_{\sigma,\nu,c}(\theta_j)$.
Moreover, these eigenvalues asymptotically separate from the rest of the spectrum since 
(with the conventions that $\lambda_0(M_N)=+\infty$ and $\lambda_{N+1}(M_N)=-\infty$)
there exists  $ \delta_0 >0$ such that 
\noindent almost surely for all large N, \begin{equation}\label{sepraj}\lambda_{n_{j-1}}(M_N) > \rho _{\theta _j} + \delta_0 \, \mbox{~and~} \, \lambda_{n_{j-1}+k_j +1}(M_N) < \rho _{\theta _j} - \delta_0 
.\end{equation}
\end{theorem}
\begin{remark}\label{gege} Note that  Theorems \ref{inclusionth}  and \ref{ThmASCV} were established in
 \cite{MC2014}  for $A_N$ as \eqref{diagonale} below and with ${\cal S}\cup \{0\}$ instead of $ {\cal S}$ but they  hold true as stated above and  in the more  general framework of this paper. Indeed, these extensions can be obtained sticking to the proof of  the corresponding results in \cite{MC2014} but using the new versions of \cite{BaiSilver}
 and  of the exact separation phenomenon of \cite{MC2014} which are presented in  the Appendix A of the present paper.
\end{remark}
\noindent 
 The aim of this paper is  to study  how the  eigenvectors corresponding to the outliers of $M_N$ project onto those corresponding to the spikes $\theta_i$'s.
Note that there are some pionneering results investigating the  eigenvectors corresponding to the outliers of 
finite rank perturbations of classical random matricial models: \cite{Paul} in the real Gaussian sample covariance
matrix setting, and \cite{BGRao09,BGRao10} dealing with finite rank additive or multiplicative
perturbations of unitarily invariant matrices. For a general perturbation, dealing
with sample covariance matrices, S. P\'ech\'e and O. Ledoit \cite{PL} introduced a
tool to study the average behaviour of the eigenvectors but it seems that this
did not allow them to focus on the eigenvectors associated with the eigenvalues
that separate from the bulk. It turns out that further studies \cite{MCJTP,BBCF15} point
out that the angle between the eigenvectors of the outliers of the deformed
model and the eigenvectors associated to the corresponding original spikes is
determined by Biane-Voiculescu's subordination function. For the model investigated in this paper, such a free interpretation holds but we choose not to develop this free probabilistic point of view in this paper and we refer the reader to the paper \cite{CDM}.
Here is  the main result of the paper.
\begin{theorem}\label{cvev1p1}
 Let  $\theta_j$ be in $\Theta_{\sigma,\nu,c}$ (defined in \eqref{defTheta})
and denote by $n_{j-1}+1, \ldots , n_{j-1}+k_j$ the descending ranks of $\alpha_j(N)$ among the eigenvalues of $A_NA_N^*$.
 Let    
 $\xi(j)$ be  a normalized  eigenvector of $M_N$  relative to one of the eigenvalues  $(\lambda_{n_{j-1}+q}(M_N)$, $1\leq q\leq k_j)$.
Denote by $\|\cdot \|_2$ the Euclidean norm on $\mathbb{C}^n$. Then, almost surely 
\begin{itemize}
\item[(i)] $\displaystyle{\lim_{N\rightarrow +\infty}\left\| P_{\mbox{Ker~}(\alpha_j(N) I_N-A_NA_N^*)}\xi(j)\right\|^2_2  = \tau(\theta_j)}$\\
 ~~

\noindent where  \begin{equation}\label{tau}\tau(\theta_j)= \frac{1-\sigma^2c g_{\mu_{\sigma,\nu,c}}(\rho_{\theta_j})}{\omega_{\sigma,\nu,c}'(\rho_{\theta_j})}=\frac{ \Phi_{\sigma,\nu,c}'({\theta_j})}{1+ \sigma^2 cg_\nu(\theta_j)}\end{equation}
\item[(ii)] for any  $\theta_i$ in $\Theta_{\sigma,\nu,c}\setminus\{\theta_j\}$, 
$$\displaystyle{\lim_{N\rightarrow +\infty}\left\| P_{ \mbox{Ker~}(\alpha_i(N)  I_N-A_NA_N^*)}\xi(j)\right\|_2 = 0.}$$

\end{itemize}
\end{theorem}
 The sketch of the  proof of Theorem \ref{cvev1p1} follows the 
analysis  of \cite{MCJTP} as explained in Section 2. In Section 3, we prove a universal result allowing to  reduce  the study to estimating expectations of Gaussian resolvent entries carried on Section 4. In Section 5, we explain how to deduce  Theorem \ref{cvev1p1} from the previous  Sections.
In an Appendix A, we present alternative versions on the one hand of the result in \cite{BaiSilver} about the lack of eigenvalues outside the support of the deterministic equivalent measure, and, on the other hand, of the result in \cite{MC2014} about the exact separation phenomenon. These new versions deal with random variables whose imaginary and real parts are independent but remove the technical assumptions ((1.10) and ``$b_1>0$'' in Theorem 1.1 in \cite{BaiSilver} and ``$\omega_{\sigma,\nu,c}(b)>0$'' in Theorem 1.2 in \cite{MC2014}). This allows us to claim that Theorem \ref{ThmASCV} holds in our context (see Remark \ref{gege}).
 Finally, we present, in an Appendix B, some technical lemmas that are used throughout the paper.

\section{Sketch of the proof}
Throughout the paper, for any $m\times p$ matrix $B$, $(m,p)\in {\mathbb{N}}^2$, we will denote by $\Vert B\Vert$  the largest singular value of $B$, and  by $\Vert B\Vert_2=\{Tr (BB^*)\}^{\frac{1}{2}}$ its Hilbert-Schmidt norm.\\
 The proof of Theorem \ref{cvev1p1} follows the 
analysis in two steps of \cite{MCJTP}.
\newline\noindent{\bf Step A.} First, we shall prove that,
   for any  orthonormal system $(\xi_1,\cdots,\xi_{k_j})$ of eigenvectors associated to the $k_j$ eigenvalues $\lambda_{n_{j-1}+q}(M_N)$, $1\leq q\leq k_j$, the following convergence holds  almost surely: $\forall l =1,\ldots,J$, 
\begin{equation}\label{5.4}
\sum_{p=1}^{k_j}\left\| P_{\ker(\alpha_l (N) I_N-A_NA_N^*)}\xi_p
\right\|^2_2   \rightarrow_{N \rightarrow +\infty} \frac{k_j\delta_{jl} (1-\sigma^2 c g_{\mu_{\sigma,\nu,c}}(\rho_{\theta_j}))}{\omega_{\sigma,\nu,c}'(\rho_{\theta_j})}. 
\end{equation}
 Note that  for any smooth functions $h$ and $f$ on $\mathbb{R}$, if $v_1,\ldots, v_n$  are eigenvectors associated to $\lambda_1(A_NA_N^*), \ldots,\lambda_n(A_NA_N^*)$ and  $w_1,
\ldots, w_n$ are eigenvectors associated to 
$\lambda_1(M_N), \ldots, \lambda_n(M_N)$, one can easily check that
\begin{equation}\label{equavec}
{\rm Tr} \left[h(M_N) f(A_NA_N^*)\right] =\sum_{m,p=1}^n h(\lambda_p(M_N)) f(\lambda_m(A_NA_N^*)) \vert 
\langle v_m,w_p \rangle \vert^2.
\end{equation}
Thus, since $\alpha_l(N)$ on one hand  and the  $k_j$ eigenvalues of $M_N$ in $(\rho_{\theta_j} 
-\varepsilon
,\rho_{\theta_j}+\varepsilon
)$ (for $\epsilon$ small enough) on the other hand, asymptotically separate from the rest of the spectrum of 
respectively $A_NA_N^*$ and $M_N$, a fit choice of $h$ and $f$ will allow the  study of the restrictive 
sum $\sum_{p=1}^{k_j}\left\| P_{\ker(\alpha_l(N) I_N-A_NA_N^*)} \xi_p
\right\|^2_2$.
Therefore proving (\ref{5.4}) is reduced to the study of the asymptotic 
behaviour of  ${\rm Tr}\left[h(M_N)f(A_NA_N^*)\right]$ for some  functions 
$f$ and $h$ respectively concentrated on a neighborhood of $\theta_l$ and $\rho_{\theta_j}$.\\

\noindent{\bf Step B:} In the second, and final, step, we shall use a perturbation
argument identical to the one used in \cite{MCJTP} to reduce the problem to the case of a 
spike with multiplicity one, case that follows trivially from Step A.\\

Step B closely follows the lines of \cite{MCJTP} whereas Step A  requires substantial work. 
We first reduce the investigations to the mean Gaussian case by  proving  the following.
\begin{proposition}\label{compgaunogau} Let $X_N$ as defined in Section \ref{no}. Let ${\cal G}_N = [{\cal G}_{ij}]_{1\leq i\leq n, 1\leq j\leq N}$ be a $n\times N$ random matrix with i.i.d. standard complex normal entries.
Let $h$ be a  function  in $\cal C^\infty (\R, \R)$ with compact support,  and $\Gamma_N$ be  a $n\times n $ Hermitian matrix
such that \begin{equation} \sup_{n,N} \Vert \Gamma_N \Vert<\infty \text{~and~} \sup_{n,N} \rm{rank} (\Gamma_N)  <\infty.\end{equation} 
Then almost surely, \\

$\Tr \left(h\left(\left(\sigma \frac{X_N}{\sqrt{N}}+A_N\right)\left(\sigma \frac{X_N}{\sqrt{N}}+A_N\right)^*\right) \Gamma_N\right)$ $$-\mathbb{E}\left(\Tr \left[h\left(\left(\sigma \frac{{\cal G}_N}{\sqrt{N}}+A_N\right)\left(\sigma \frac{{\cal G}_N}{\sqrt{N}}+A_N\right)^*\right) \Gamma_N\right] \right)\rightarrow_{N \rightarrow +\infty} 0.$$

\end{proposition}
\noindent The asymptotic behaviour of $\mathbb{E}\left(\Tr \left[h\left(\left(\sigma\frac{{\cal G}_N}{\sqrt{N}}+A_N\right)\left(\sigma \frac{{\cal G}_N}{\sqrt{N}}+A_N\right)^*\right) f(A_NA_N^*)\right] \right)$ can be deduced, by using the bi-unitarily invariance of the distribution of $ {\cal G}_N$, from the following Proposition \ref{estimfonda} and  Lemma \ref{approxpoisson}.

\begin{proposition}\label{estimfonda} Let ${\cal G}_N = [{\cal G}_{ij}]_{1\leq i \leq n, 1\leq j\leq N}$ be a $n\times N$ random matrix with i.i.d. complex standard normal entries. Assume that $A_N$ is such that 
\begin{equation}\label{diagonale} A_N=\begin{pmatrix}  d_1(N) ~~~~~~~~~~~~~~~~~~~~(0)\\
~~~(0)\\ ~~~~~~~~~~\ddots~~~~~~~~~~~~~( 0)\\ ~(0)~~~~~~~~~~~~~~~~~~~~\\  ~~~~~~~~~~~~~~~~~d_{n}(N)~~~ ( 0  )    \end{pmatrix} \end{equation} where  $n=n(N)$, $n \leq N$, $c_N=n/ N\rightarrow_{N\rightarrow +\infty} c \in]0;1]$, for $i=1,\ldots,n$, $d_i(N) \in \mathbb{C}$, $\sup_N \max_{i=1,\ldots,n} \vert d_i(N)\vert <+\infty$ and $\frac{1}{n} \sum_{i=1}^n \delta_{\vert d_i(N)\vert^2}$ weakly converges to a compactly supported probability measure $\nu$ on $\mathbb{R}$ when $N$ goes to infinity.
Define for all  $z\in \mathbb{C}\setminus \mathbb{R}$, $$G^{\cal G}_N(z) =\left (zI - \left(\sigma \frac{{\cal G}_N}{\sqrt{N}}+A_N\right)\left(\sigma \frac{{\cal G}_N}{\sqrt{N}}+A_N\right)^*\right)^{-1}.$$
 Define for any $q=1,\ldots,n$, \begin{equation}\label{qq}\gamma_q(N) =(A_NA_N^*)_{qq} =\vert d_q(N)\vert^2. \end{equation}  There is a polynomial $P$ with nonnegative coefficients,  a sequence $(u_N)_N$ of nonnegative
real numbers converging to zero when $N$ goes to infinity 
and some nonnegative real number $l$, 
such that for any $(p,q)$ in $\{1, \ldots,n\}^2$, for all  $z\in \mathbb{C}\setminus \mathbb{R}$,
\begin{equation} \label{premier}
\mathbb{E} \left(\left( G^{\cal G}_N(z)\right)_{pq}\right) = \frac{1- \sigma^2 cg_{\mu_{\sigma,\nu,c}}(z)}{\omega_{\sigma, \nu, c}(z) -\gamma_q(N)} \delta_{pq}
+\Delta_{p,q,N}(z),
\end{equation}
with $$\left| \Delta_{p,q,N} (z)\right| \leq (1+\vert z\vert)^l P(\vert \Im z \vert^{-1})u_N.$$

\end{proposition}
 
\section{Proof of Proposition \ref{compgaunogau}}
In the following, we will denote by $o_{C}(1)$ any  deterministic   sequence of positive real numbers  depending on  the parameter $C$  and converging for each fixed $C$  to zero when $N$ goes to infinity. The aim of this section is to prove Proposition
\ref{compgaunogau}.\\

\noindent Define for any  $C>0$,  \begin{eqnarray}Y_{ij}^C &=&\Re  X_{ij}\1_{\vert  \Re X_{ij} \vert \leq C} - \mathbb{E}\left( \Re  X_{ij}\1_{\vert \Re  X_{ij} \vert \leq C} \right) \nonumber \\&&+ \sqrt{-1} \left\{
\Im  X_{ij}\1_{\vert  \Im  X_{ij} \vert \leq C} - \mathbb{E}\left( \Im  X_{ij}\1_{\vert \Im  X_{ij} \vert \leq C} \right) \right\}.\label{ycdef}\end{eqnarray}

\noindent Set $$\theta^*=\sup_{(i,j)\in \mathbb{N}^2}\mathbb{E}(\vert X_{ij}\vert^3)<+\infty. $$
We have 
\begin{eqnarray*}\mathbb{E} \left( \vert   X_{ij}-Y_{ij}^C\vert^2 \right)&=&\mathbb{E} \left( \vert \Re   X_{ij}\vert^2 \1_{\vert \Re  X_{ij} \vert > C} \right)
+ \mathbb{E} \left( \vert \Im  X_{ij}\vert^2 \1_{\vert \Im  X_{ij} \vert > C} \right) \\&&
-\left\{ \mathbb{E} \left( \Re  X_{ij} \1_{\vert \Re  X_{ij} \vert > C} \right)\right\}^2-\left\{ \mathbb{E} \left( \Im   X_{ij} \1_{\vert \Im  X_{ij} \vert > C} \right)\right\}^2\\& \leq & \frac{\mathbb{E} \left(\vert \Re  X_{ij} \vert^3\right)+\mathbb{E} \left(\vert \Im  X_{ij} \vert^3\right)}{C}
\end{eqnarray*}
so that $$\sup_{i\geq 1,j\geq 1}\mathbb{E} \left( \vert  X_{ij}-Y_{ij}^C\vert^2 \right) \leq \frac{2\theta^*}{C}.$$

\noindent Note that 
\begin{eqnarray*}1 - 2 \mathbb{E} \left( \vert \Re  Y_{ij}^C\vert^2 \right) &=& 1-2 \mathbb{E}   \left\{\left(\Re   X_{ij} \1_{\vert \Re X_{ij} \vert \leq  C} -  \mathbb{E} \left(  \Re    X_{ij} \1_{\vert  \Re X_{ij} \vert \leq  C}\right)\right)^2\right\}
\\&=& 2\left[ \frac{1}{2} -\mathbb{E} \left( \vert  \Re  X_{ij}\vert^2 \1_{\vert \Re X_{ij} \vert \leq  C} \right)\right] 
+2\left\{ \mathbb{E} \left( \Re X_{ij} \1_{\vert \Re  X_{ij} \vert \leq  C} \right)\right\}^2\\&=& 
2 \mathbb{E} \left( \vert \Re   X_{ij}\vert^2 \1_{\vert\Re  X_{ij} \vert > C} \right)+ 2\left\{\mathbb{E} \left(  \Re  X_{ij} \1_{\vert \Re  X_{ij} \vert > C}\right) \right\}^2, \end{eqnarray*}
so that $$\sup_{i\geq 1, j\geq 1}\vert 1 - 2\mathbb{E} \left( \vert \Re  Y_{ij}^C\vert^2 \right) \vert \leq \frac{4\theta^*}{C}.$$
Similarly $$\sup_{i\geq 1,j\geq 1}\vert 1 - 2\mathbb{E} \left( \vert \Im  Y_{ij}^C\vert^2 \right) \vert \leq \frac{4\theta^*}{C}.$$
Let us assume that $C>8\theta^*.$ Then, we have 
$$\mathbb{E} \left( \vert \Re  Y_{ij}^C\vert^2 \right)> \frac{1}{4} \; \mbox{and}\; \mathbb{E} \left( \vert \Im  Y_{ij}^C\vert^2 \right)> \frac{1}{4}.$$
Define for any $C> 8\theta^*$, $X^C=(X^C_{ij})_{1\leq i \leq n; 1\leq j \leq N},$ where for any $1\leq i \leq n, 1\leq j \leq N$,
 \begin{equation}\label{defxc}{X}_{ij}^C =\frac{\Re  Y_{ij}^C}{\sqrt{2\mathbb{E} \left( \vert \Re Y_{ij}^C\vert^2 \right)}} +\sqrt{-1}  \frac{\Im  Y_{ij}^C}{\sqrt{2\mathbb{E} \left( \vert \Im  Y_{ij}^C\vert^2 \right)}}.\end{equation}

Let ${\cal G} = [{\cal G}_{ij}]_{1\leq i\leq n, 1\leq j \leq N}$ be a $n \times N$ random matrix with i.i.d. standard  complex normal entries,  independent from $X_N$, and define 
for any $\alpha>0$,
$$X^{\alpha,C}= \frac{ X^C +\alpha {\cal G}}{\sqrt{1+\alpha^2}}.$$
 Now, for any $n\times N$ matrix $B$, let us introduce the $(N+n) \times (N+n)$ matrix $${\cal M}_{N+n}(B) =\left( \begin{array}{ll} 0_{n\times n}~~~ B+A_N\\ B^* +A_N^*~~~ 0_{N\times N} \end{array} \right).$$
Define  for any $z\in \mathbb{C}\setminus \mathbb{R}$, $$\tilde G(z) = \left( z I_{N+n} - {\cal M}_{N+n}\left(\sigma \frac{X_N}{\sqrt{N}}\right)\right)^{-1},$$ and 
$$\tilde G^{\alpha,C}(z) = \left( z I_{N+n} - {\cal M}_{N+n}\left(\sigma \frac{X^{\alpha,C}}{\sqrt{N}}\right)\right)^{-1} .$$
Denote by $\mathbb{U}(n+N)$ the set of unitary $(n+N)\times (n+N)$ matrices.
We first establish the following approximation result.

\begin{lemme}\label{approximation}
\noindent There exist some positive deterministic functions $u$ and $v$ on $[0,+\infty[$ such that $\lim_{C\rightarrow +\infty}  u(C) =0$ and $\lim_{\alpha \rightarrow 0} v(\alpha)=0,$  and a polynomial $P$  with nonnegative coefficients   such that  for any $\alpha$ and $ C>8\theta^*$, we have that \\
$\bullet$ almost surely, for all large N,\\

$\displaystyle{ \sup_{U\in \mathbb{U}(n+N)}\sup_{(i,j)\in \{1,\ldots,n+N\}^2}\sup_{z\in \mathbb{C}\setminus \mathbb{R} } |\Im z |^{2}  \left| (U^*\tilde{G}^{\alpha,C}(z)U)_{ij}-  (U^*\tilde{{G}}(z)U)_{ij}\right|}$
\begin{equation} \label{EGC}   \leq  u(C)+v(\alpha), 
\end{equation}
$\bullet$  for all large $N$, 
\\$\displaystyle{  \sup_{U\in \mathbb{U}(n+N)} \sup_{(i,j)\in  \{1,\ldots,n+N\}^2} \sup_{z\in \mathbb{C}\setminus \mathbb{R} } \frac{1}{  P(|\Im z |^{-1})  }\left|  \mathbb{E} \left( (U^*\tilde{G}^{\alpha,C}(z)U)_{ij}-  (U^*\tilde{{G}}(z)U)_{ij}\right)\right|}$\begin{equation} \label{EGCalpha} \leq  u(C)+v(\alpha)+  o_{C}(1).
\end{equation}

\end{lemme}
\begin{proof}

Note that \begin{eqnarray*} {X}_{ij}^C-Y_{ij}^C&=& \Re  X_{ij}^C \left( 1-\sqrt{2}  \mathbb{E} \left( \vert \Re  Y_{ij}^C\vert^2 \right)^{1/2}\right)
+\sqrt{-1} \Im  X_{ij}^C \left( 1-\sqrt{2}  \mathbb{E} \left( \vert \Im  Y_{ij}^C\vert^2 \right)^{1/2}\right)
\\&=&\Re  X_{ij}^C\frac{1 - 2\mathbb{E} \left( \vert \Re  Y_{ij}^C\vert^2 \right) }{1 + \sqrt{2}\mathbb{E} \left( \vert \Re  Y_{ij}^C\vert^2 \right)^{1/2}}+
\sqrt{-1} \Im  X_{ij}^C\frac{1 - 2\mathbb{E} \left( \vert \Im  Y_{ij}^C\vert^2 \right) }{1 + \sqrt{2}\mathbb{E} \left( \vert \Im  Y_{ij}^C\vert^2 \right)^{1/2}}.\end{eqnarray*}  Then, 
$$\left\{\sup_{(i,j)\in \mathbb{N}^2}\mathbb{E} \left( \vert  X^C_{ij}-Y_{ij}^C\vert^2 \right)\right\}^{1/2} \leq \frac{4\theta^*}{C},
\mbox{~and~}\sup_{(i,j)\in \mathbb{N}^2}\mathbb{E} \left( \vert  X_{ij}^C-Y_{ij}^C\vert^3 \right) <\infty.$$
It is straightfoward to see, using  Lemma \ref{lem0}, that for any unitary $(n+N)\times (n+N)$ matrix $U$,\\
 
$
\left| (U^*\tilde{G}^{\alpha,C}(z)U)_{ij}-  (U^*\tilde{{G}}(z)U)_{ij}\right|$
\begin{eqnarray} &\leq & \frac{\sigma}{\vert \Im z \vert^2} \left\| \frac{X_N-{X}^{\alpha,C}}{\sqrt{N}} \right\| \nonumber \\
&\leq &  \frac{\sigma}{\vert \Im z \vert^2} \left\{  \left\| \frac{X_N-Y^C}{\sqrt{N}} \right\| + \left\| \frac{{X}^C-Y^C}{\sqrt{N}} \right\|\right. \nonumber\\&&  \left. +  \left(1- \frac{1}{\sqrt{1+\alpha^2}}\right)   \left\| \frac{ X^C}{\sqrt{N}} \right\| +\alpha  \left\| \frac{{\cal G}}{\sqrt{N}} \right\| \right\} \label{inegu}.
\end{eqnarray}
From Bai-Yin's theorem  (Theorem 5.8 in \cite{BaiSil06}) , we have
$$\left\| \frac{{\cal G}}{\sqrt{N}} \right\|=2+o(1).$$ Applying Remark \ref{2.1} to  the $(n+N)\times (n+N)$ matrix  $\tilde B=
\left( \begin{array}{ll} 0_{n\times n}~~~ B\\ B^*~~~ 0_{N\times N} \end{array} \right)$ for $B\in \{X_N-Y^C,  X^C-Y^C,X^C\}$ (see also Appendix B of \cite{CSBD}), we have that almost surely 
$$ \limsup_{N\rightarrow +\infty}\left\| \frac{{X^C}}{\sqrt{N}} \right\|\leq 2 \sqrt{2} ,~\limsup_{N\rightarrow +\infty}\left\| \frac{{X}^C-Y^C}{\sqrt{N}} \right\|\leq \frac{8\sqrt{2}\theta^*}{C},\;$$ and  $$ \limsup_{N\rightarrow +\infty}\left\| \frac{{X_N}-Y^C}{\sqrt{N}} \right\|\leq 4 \sqrt{ \frac{\theta^*}{C}}.
$$
Then, 
 (\ref{EGC}) readily follows.

\noindent Let us introduce $$\Omega_{N,C}=\left\{ \left\| \frac{{\cal G}}{\sqrt{N}} \right\| \leq 4, \left\| \frac{X^C}{\sqrt{N}} \right\| \leq 4, \left\| \frac{X_N-Y^C}{\sqrt{N}} \right\| \leq  8 \sqrt{ \frac{\theta^*}{C}}, \left\| \frac{{X}^C-Y^C}{\sqrt{N}} \right\|\leq \frac{16\theta^*}{C} \right\}.$$
Using \eqref{inegu}, we have\\

\noindent  $\left|  \mathbb{E} \left( (U^*\tilde{{G}}^{\alpha,C}(z)U)_{ij}-(U^*\tilde{G}(z)U)_{ij}\right)\right|$  \begin{eqnarray*}&\leq & 
\frac{4\sigma }{\vert \Im z \vert^2} \left[  2 \sqrt{ \frac{\theta^*}{C}}+ \frac{4\theta^*}{C}+
 \alpha +\left( 1-\frac{1}{\sqrt{1+\alpha^2}}\right) \right] \\& &+ \frac{2}{\vert \Im z \vert} \mathbb{P}(\Omega_{N,C}^c).
\end{eqnarray*}
Thus (\ref{EGCalpha}) follows.
\end{proof}
Now, Lemma \ref{approxpoisson},   Lemma \ref{approximation} and Lemma \ref{HT} readily yields the following approximation lemma.
\begin{lemme}\label{approxCalpha}
Let $h$  be  in $\cal C^\infty (\R, \R)$  with compact support and $\tilde \Gamma_N$ be  a $(n+N)\times (n+N)$ Hermitian matrix such that 
such that \begin{equation} \sup_{n,N} \Vert \tilde  \Gamma_N \Vert<\infty \text{~and~} \sup_{n,N} \rm{rank} (\tilde \Gamma_N)  <\infty.\end{equation}   Then, there exist  some deterministic functions $u$ and $v$ on $[0,+\infty[$ such that $\lim_{C\rightarrow +\infty}  u(C) =0$ and $\lim_{\alpha \rightarrow 0} v(\alpha)=0,$ such that for all  $C>0$, $\alpha>0$, we have almost surely for all large N,
\begin{equation}\label{diff}  \left|\Tr \left[ h\left(({\cal M}_{N+n}\left(\frac{X^{\alpha,C}}{\sqrt{N}}\right)\right) \tilde \Gamma_N\right]
- \Tr \left[  h\left(({\cal M}_{N+n}\left(\frac{X_N}{\sqrt{N}}\right)\right) \tilde \Gamma_N\right]\right| \leq   a^{(1)}_{C,\alpha},\end{equation}
and for all large $N$, 
\begin{equation}\label{Ediff}  \left|\mathbb{E}\Tr \left[ h\left(({\cal M}_{N+n}\left(\frac{X^{\alpha,C}}{\sqrt{N}}\right)\right) \tilde \Gamma_N\right]
- \mathbb{E}\Tr \left[  h\left(({\cal M}_{N+n}\left(\frac{X_N}{\sqrt{N}}\right)\right) \tilde  \Gamma_N\right]\right| \leq  a^{(2)}_{C,\alpha,N},\end{equation}
where   $$a^{(1)}_{C,\alpha}= u(C)+v(\alpha),\; a^{(2)}_{C,\alpha,N}= u(C)+v(\alpha) + o_{C}(1).$$
\end{lemme}
\noindent Note that    
 the distributions of the independent random variables
$\Re (X_{ij}^{\alpha,C})$, $\Im (X_{ij}^{\alpha,C})$  are all    a convolution of a centred   Gaussian distribution with some variance $v_\alpha $,  with  some  law with bounded support
in a ball of some radius $R_{C,\alpha}$; thus, according to Lemma \ref{zitt},    
they satisfy a  Poincar\'e inequality with some common  constant $ C_{PI}(C,\alpha)$  and therefore so does their product (see the Appendix B).
An important consequence of the Poincar\'e inequality is the following concentration result.
\begin{lemme}\label{Herbst}{ Lemma 4.4.3 and Exercise 4.4.5 in \cite{AGZ09} or Chapter 3 in  \cite{Ledoux01}. } There exists $K_1>0$ and $K_2>0$ such that for any 
probability measure  $\mathbb{P}$  on $\mathbb{R^M}$ which satisfies a Poincar\'e inequality with constant $C_{PI}$, and  for any  Lipschitz function $F$  on $\mathbb{R}^M$ with Lipschitz constant $\vert F \vert_{Lip}$, we have 
$$\forall \epsilon> 0,  \, \mathbb{P}\left( \vert F-\E_{\mathbb{P}}(F) \vert > \epsilon \right) \leq K_1 \exp\left(-\frac{\epsilon}{K_2 \sqrt{C_{PI}} \vert F \vert_{Lip}}\right).$$
\end{lemme}
In order to apply  Lemma \ref{Herbst}, we  need the following preliminary lemmas.
\begin{lemme}\label{extlipschitz}(see Lemma 8.2 \cite{MCJTP})
Let $f$ be a real $C_{\cal L}$-Lipschitz function on $\mathbb{R}$. Then its extension on the $N\times N$
Hermitian matrices is $C_{\cal L}$-Lipschitz with respect to the Hilbert-Schmidt  norm. 
\end{lemme}
\begin{lemme}\label{Lipschitz} Let $\tilde \Gamma_N $ be a $(n+N)\times ( n+N)$ matrix and  $h$ be a real Lipschitz function on $\mathbb{R}$.
For any $n \times N$  matrix $B$,
$$\left\{\left(\Re B(i,j),~ \Im B(i,j)\right)_{1\leq i \leq n, 1 \leq j \leq N}\right\} \mapsto Tr \left[  h\left(({\cal M}_{N+n}\left(B\right)\right) \tilde \Gamma_N\right]$$
\noindent  is Lipschitz with constant bounded by $\sqrt{2} \left\| \tilde \Gamma_N \right\|_2  \Vert h \Vert_{Lip}$.

\end{lemme}
\begin{proof}

$ \left| \Tr \left[h ({\cal M}_{N+p}(B)) \tilde \Gamma_N\right]- \Tr \left[ h ({\cal M}_{N+p}(B^{'}))\tilde \Gamma_N\right]\right|$
\begin{eqnarray}
 \nonumber\\ &\leq &
 \left\| \tilde \Gamma_N \right\|_2 \left\| h ({\cal M}_{N+p}(B))-
h ({\cal M}_{N+p}(B^{'}))\right\|_2 \nonumber\\
 &\leq &  \left\| \tilde \Gamma_N \right\|_2 \left\|h \right\|_{Lip} \left\| {\cal M}_{N+p}(B)-{\cal M}_{N+p}(B^{'}) \right\|_2.\label{lip1}
 \end{eqnarray}
 \noindent where we used Lemma \ref{extlipschitz} in the last line.
 Now, \begin{equation}\left\| {\cal M}_{N+p}(B)-{\cal M}_{N+p}(B^{'}) \right\|^2_2= 2 \left\| B-B^{'}\right\|^2_2.\label{lip2}\end{equation}
 
\noindent Lemma \ref{Lipschitz} readily follows from (\ref{lip1}) and (\ref{lip2}). 
\end{proof}
\begin{lemme}\label{inegconc} Let $\tilde \Gamma_N $ be a $(n+N)\times ( n+N)$ matrix such that $  \sup_{N,n} \left\| \tilde \Gamma_N \right\|_2 \leq K$. Let $h$ be a real Lipschitz function on $\mathbb{R}$.
The random variable $F_N={\rm Tr} \left[ h\left( {\cal M}_{N+p}\left( \frac{X^{\alpha,C}}{\sqrt{N}}\right) \right) \tilde \Gamma_N\right]$
satisfies the following concentration inequality
$$\forall \epsilon> 0, \, \mathbb{P}\left( \vert F_N-\E(F_N) \vert > \epsilon \right) \leq K_1 \exp\left(-\frac{\epsilon \sqrt{ N}}{K_2(\alpha,C)  K \Vert h \Vert_{Lip}}\right),$$
for some postive real numbers $K_1$ and $K_2(\alpha,C)$.

\end{lemme}
\begin{proof}  Lemma \ref{inegconc} follows from
 Lemmas \ref{Lipschitz} and \ref{Herbst} and basic facts on Poincar\'e inequality recalled at the end of the Appendix $B$.
\end{proof}

\noindent By Borel-Cantelli's Lemma, we readily deduce from the above Lemma the following
\begin{lemme}\label{concentration} 
 Let $\tilde \Gamma_N $ be a $(n+N)\times ( n+N)$ matrix such that $ \sup_{N,n} \left\| \tilde \Gamma_N \right\|_2 \leq K$. Let $h$ be a real ${\cal C}^1$- function with compact support on $\mathbb{R}$. \\

${\rm Tr} \left[ h\left( {\cal M}_{N+p}\left( \sigma \frac{X^{\alpha,C}}{\sqrt{N}}\right) \right)\tilde \Gamma_N\right]- \mathbb{E}\left[ {\rm Tr} \left[ h\left( {\cal M}_{N+p}\left( \sigma \frac{X^{\alpha,C}}{\sqrt{N}}\right) \right) \tilde \Gamma_N\right]\right]$ \begin{equation}\label{concentre}~~~~~~~~~~~~~~~~~~~~~~~~~~~~~~~~~~~~~~~~~~~~~~~~~~~~~~~~~~~~~~~~~~~~~~~~~~~~~~~~~~~\stackrel{a.s}{\longrightarrow}_{N\rightarrow +\infty}0.
\end{equation}
\end{lemme}

Now, we will  establish a comparison result with the Gaussian case for the mean values by using the following lemma (which is an  extension of Lemma \ref{gaussian} below to the non-Gaussian case) as initiated by \cite{KKP96} in Random Matrix Theory.
\begin{lemme} \label{IPP}
Let $\xi$ be a real-valued random variable such that  $\E(\vert \xi \vert ^{p+2}) < \infty $. 
Let $\phi $ be a function from $\R$ to $\C$ such that the first $p+1$ derivatives are continuous and bounded.
Then, 
\begin{equation}\label{IPP2}
\E (\xi \phi (\xi )) = \sum_{a=0}^p \frac{\kappa _{a+1}}{a!}\E (\phi ^{(a)}(\xi )) + \epsilon , 
\end{equation} 
where $\kappa _{a}$ are the cumulants of $\xi $, 
$\vert \epsilon \vert \leq K \sup_t \vert \phi ^{(p+1)}(t)\vert \E (\vert \xi \vert ^{p+2})$, 
$K$ only depends on $p$.
\end{lemme}
\begin{lemme}\label{Chatterjee}
 Let ${\cal G}_N = [{\cal G}_{ij}]_{1\leq i\leq n, 1\leq j\leq N}$ be a $n\times N$ random matrix with i.i.d. complex $N(0,1)$  Gaussian entries.
Define $$\tilde G^{{\tiny {\cal G}}}(z)= \left( zI_{N+n}- {\cal M}_{N+n}\left(\sigma \frac{{\cal G}_N}{\sqrt{N}}\right) \right)^{-1}$$
for any $z\in \mathbb{C}\setminus \mathbb{R}.$
There exists  a polynomial $P$ with nonnegative coefficients  such that for all large $N$, for  any $(i,j)\in \{1,\ldots,n+N\}^2$,  for any $z\in \mathbb{C}\setminus \mathbb{R}$, for any unitary $(n+N)\times (n+N)$ matrix $U$,
\begin{equation}\label{gau} \left|   \mathbb{E} \left[(U^*\tilde  G^{{\tiny {\cal G}}}(z)U)_{ij}\right]- \mathbb{E}\left[(U^*\tilde G(z)U)_{ij}\right]\right| \leq \frac{1}{\sqrt{N}}P(\left|\Im z \right|^{-1}).\end{equation}
Moreover,  for any $(N+n)\times (N+n)$ matrix $\tilde \Gamma_N$ such that  \begin{equation}\label{gauG} \sup_{n,N} \Vert \tilde  \Gamma_N \Vert<\infty \text{~and~} \sup_{n,N} \rm{rank} (\tilde \Gamma_N)  <\infty,\end{equation}
and any function $h$ in $\cal C^\infty (\R, \R)$ with compact support, there exists some constant $K>0$ such that, for any large N, \\

\noindent $ \left|  \mathbb{E}\left[ {\rm Tr} \left[ h\left( {\cal M}_{N+n}\left( \sigma \frac{X_N}{\sqrt{N}}\right) \right)\tilde \Gamma_N\right]\right]- 
 \mathbb{E}\left[ {\rm Tr} \left[ h\left( {\cal M}_{N+n}\left( \sigma \frac{{\cal G}_N}{\sqrt{N}}\right) \right)\tilde \Gamma_N\right]\right] \right|
$
\begin{equation}\label{comparaison}~~~~~~~~~~~~~~~~~~~~~~~~~~~~~~~~~~~~~~~~~~~~~~~~~~~~~~~~~~~~~~~~~~~~
\leq \frac{K}{\sqrt{N}}.\end{equation}
\end{lemme}

\begin{proof}
We follow the approach of \cite{PS} chapters 18 and 19 consisting in  introducing an interpolation matrix 
$X_N(\alpha)= \cos \alpha  X_N + \sin \alpha  {\cal G}_N$ for any $\alpha$ in $[0;\frac{\pi}{2}]$ and the corresponding resolvent matrix $\tilde G(\alpha,z)= \left( zI_{N+n}- {\cal M}_{N+n}\left(\sigma \frac{X_N(\alpha)}{\sqrt{N}}\right) \right)^{-1}$
for any $z\in \mathbb{C}\setminus \mathbb{R}.$ 
We have, for any $(s,t)\in \{1,\ldots,n+N\}^2$,
$$\mathbb{E}\tilde G^{{\tiny {\cal G}}}_{st}(z)- \mathbb{E} \tilde  G_{st}(z)=\int_0^{\frac{\pi}{2}} \mathbb{E} \left( \frac{ \partial }{\partial \alpha} \tilde G_{st}(\alpha,z)\right) d\alpha $$ with 
\begin{eqnarray*} \frac{ \partial }{\partial \alpha} \tilde G_{st}(\alpha,z)&= &\frac{\sigma}{2\sqrt{N}}\sum_{l=1}^n \sum_{k=n+1}^{n+N }\left\{
 \left[ \tilde G_{sl} (\alpha,z) \tilde G_{kt} (\alpha,z)+ \tilde G_{sk}(\alpha,z) \tilde G_{l t}(\alpha,z) \right]\right.\\&&~~~~~~~~~~~~~~~~~~~~~~~~~~~\times \left[ -\sin \alpha \Re X_{l(k-n)}+\cos \alpha \Re {\cal G}_{l(k-n)}\right] \\
&&\left.~~~~~~~~~~~~~~~~~~~~~~+i 
 \left[ \tilde G_{sl}(\alpha,z)  \tilde G_{kt} (\alpha,z)- \tilde G_{sk} (\alpha,z)\tilde G_{l t}(\alpha,z) \right]\right.\\&&~~~~~~~~~~~~~~~~~~~~~~~~~~~~~\left.\times \left[-\sin \alpha \Im X_{l(k-n)} +\cos \alpha \Im {\cal G}_{l(k-n)} \right]\right\}.
\end{eqnarray*}
Now, for any $ l=1,\ldots,n $ and $k=n+1, \ldots, n+N$,  using Lemma \ref{IPP} for $p=1$ and  for each random variable $\xi$  in the set  $\left\{\Re X_{l(k-n)}, \Re {\cal G}_{l(k-n)},\Im  X_{l(k-n)}, \Im  {\cal G}_{l(k-n)} \right\} $, and for each $\phi$ in the set  $$\left\{  (U^*\tilde G (\alpha ,z))_{ip}(\tilde G(\alpha,z)U)_{q j} ;
 (p,q)=(l,k) \mbox{~or~}(k,l), (i,j)\in \{1,\ldots,n+N\}^2\right\},$$
 one can easily see that there exists some constant $K>0$ such that 
$$\left| \mathbb{E} (U^*  \tilde  G^{{\tiny {\cal G}}}(z)U)_{ij}- \mathbb{E}(U^*\tilde G(z)U)_{ij} \right| \leq \frac{ K}{N^{3/2}}\sup_{Y \in {\cal H}_{n+N}(\mathbb{C})} \sup_{V\in \mathbb{U}(n+N)}  S_V(Y) $$
where ${\cal H}_{n+N}(\mathbb{C})$ denotes the set of $(n+N) \times (n+N)$ Hermitian matrices and 
 $S_V(Y)$ is a sum of a finite number independent of $N$ and $n$  of terms of the form 
\begin{equation} \label{termes}\sum_{l=1}^n \sum_{k=n+1}^{n+N } \left|\left(U^*R(Y)\right)_{ip_1}\left(R(Y)\right)_{p_2p_3}\left(R(Y)\right)_{p_4p_5}\left(R(Y)U\right)_{p_6j} \right|\end{equation}
with  $R(Y)=\left(zI_{N+n}- Y\right)^{-1}$ and  $\{p_1, \ldots,p_6\}$ contains exactly three $k$ and three $l$. \\
When $(p_1, p_6)=(k,l)$ or $(l,k)$, then, using Lemma \ref{lem0}, \\

$ \sum_{l=1}^n \sum_{k=n+1}^{n+N }\left| \left(U^*R(Y)\right)_{ip_1}\left(R(Y)\right)_{p_2p_3}\left(R(Y)\right)_{p_4p_5}\left(R(Y)U\right)_{p_6j}\right|$
\begin{eqnarray*}
&\leq& \frac{1}{\vert \Im z \vert^2} \sum_{k,l =1}^{n+N} \left|\left(U^*R(Y)\right)_{il}\left(R(Y)U\right)_{kj}\right| \\ &\leq& \frac{(N+n)}{\vert \Im z \vert^2} 
\left(\sum_{l =1}^{n+N} \left|\left(U^*R(Y)\right)_{il}\right|^2 \right)^{1/2}\left(\sum_{k =1}^{n+N} \left|\left(R(Y)U\right)_{kj}\right|^2 \right)^{1/2}\\&=&\frac{(N+n)}{\vert \Im z \vert^2} 
\left( \left(U^*R(Y)R(Y)^*U\right)_{ii} \right)^{1/2}\left( \left(U^*R(Y)^*R(Y)U\right)_{jj} \right)^{1/2}\\&\leq & \frac{(N+n)}{\vert \Im z \vert^4} 
\end{eqnarray*}
When $p_1=p_6=k$ or $ l$, 
then, using Lemma \ref{lem0}, \\

$\sum_{l=1}^n \sum_{k=n+1}^{n+N }\left| \left(U^*R(Y)\right)_{ip_1}\left(R(Y)\right)_{p_2p_3}\left(R(Y)\right)_{p_4p_5}\left(R(Y)U\right)_{p_6j}\right|$
\begin{eqnarray*}
&\leq& \frac{N+n}{\vert \Im z \vert^2} \sum_{l =1}^{n+N} \left|\left(U^*R(Y)\right)_{il}\left(R(Y)U\right)_{lj}\right| \\ &\leq& \frac{(N+n)}{\vert \Im z \vert^2} 
\left(\sum_{l =1}^{n+N} \left|\left(U^*R(Y)\right)_{il}\right|^2 \right)^{1/2}\left(\sum_{l =1}^{n+N} \left|\left(R(Y)U\right)_{lj}\right|^2 \right)^{1/2}\\&=&\frac{(N+n)}{\vert \Im z \vert^2} 
\left( \left(U^*R(Y)R(Y)^*U\right)_{ii} \right)^{1/2}\left( \left(U^*R(Y)^*R(Y)U\right)_{jj} \right)^{1/2}\\&\leq & \frac{(N+n)}{\vert \Im z \vert^4} 
\end{eqnarray*}
(\ref{gau}) readily follows.

Then by  Lemma \ref{HT}, there exists some constant $K>0$ such that, for any $N$ and $n$, for any $(i,j)\in \{1,\ldots,n+N\}^2$,  any unitary $(n+N)\times (n+N)$ matrix $U$, 
\begin{equation} \label{HT4E}\limsup_{y \rightarrow 0^+} \left| \int \left[ \mathbb{E}(U^*\tilde G(t+iy)U)_{ij}- \mathbb{E}(U^* {\tilde G^{{\tiny {\cal G}}}}(t+iy)U)_{ij}\right]  h(t) dt \right| \leq \frac{ K}{\sqrt{N}}.\end{equation}
Thus, using (\ref{avecE}) and (\ref{gauG}), we can deduce (\ref{comparaison}) from (\ref{HT4E}).
\end{proof}

The above comparison lemmas allow us to establish the following convergence result.
\begin{proposition}\label{comptilde} Let $h$ be a function  in $\cal C^\infty (\R, \R)$ with compact support and let
  $\tilde \Gamma_N $ be a $(n+N)\times ( n+N)$ matrix such that
  $ \sup_{n,N} \rm{rank} (\tilde \Gamma_N)  <\infty$ and  $ \sup_{n,N} \Vert \tilde  \Gamma_N \Vert<\infty$. Then we have that almost surely
\\

\noindent $
Tr \left[ h\left(({\cal M}_{N+n}\left(\sigma \frac{X_N}{\sqrt{N}}\right)\right) \tilde \Gamma_N\right]-  \mathbb{E}\left[Tr \left[ h\left(({\cal M}_{N+n}\left(\sigma \frac{{\cal G}_N}{\sqrt{N}}\right)\right) \tilde \Gamma_N\right] \right]$
\begin{equation}~~~~~~~~~~~~~~~~~~~~~~~~~~~~~~~~~~~~~~~~~~~~~~~~~~~~~~
{\longrightarrow}_{N\rightarrow +\infty}0.
\end{equation}
\end{proposition}
\begin{proof}
%
 Lemmas \ref{approxCalpha},  \ref{concentration} and \ref{Chatterjee} readily yield that there exist some positive deterministic functions $u$ and $v$ on $[0,+\infty[$ with $\lim_{C\rightarrow +\infty}  u(C) =0$ and $\lim_{\alpha \rightarrow 0} v(\alpha)=0,$ such that for any $C>0$ and any $\alpha >0$, almost surely 
$$\limsup_{N \rightarrow +\infty} \left| Tr \left[ h\left(({\cal M}_{N+n}\left(\sigma \frac{X_N}{\sqrt{N}}\right)\right) \tilde  \Gamma_N\right]-  \mathbb{E}\left[Tr \left[ h\left(({\cal M}_{N+n}\left(\sigma \frac{{\cal G}_N}{\sqrt{N}}\right)\right) \tilde \Gamma_N\right] \right] \right|$$
$$ \leq 
u(C) +v(\alpha).$$
The result follows by letting $\alpha$ go to zero and $C$ go to infinity.

\end{proof}
 Now, note that, for any $N \times n$ matrix $B$, for any continuous real function  function $h$  on $\R$,  and any $n\times n $ Hermitian matrix $\Gamma_N$, we have 
$$Tr \left(h\left((B+A_N)(B+A_N)^*\right) \Gamma_N\right)=Tr \left[\tilde  h\left({\cal M}_{N+n}\left(B\right)\right) \tilde  \Gamma_N\right]$$ \noindent where $\tilde h(x)=h(x^2)$ and $ \tilde \Gamma_N= \begin{pmatrix} \Gamma_N & (0)\\ (0) & (0) \end{pmatrix}$. Thus, Proposition \ref{comptilde} readily yields Proposition \ref{compgaunogau}.
\section{Proof of  Proposition \ref{estimfonda}}
The aim of this section is to prove Proposition \ref{estimfonda} which deals with Gaussian random variables.Therefore  we  assume here  that 
$A_N$ is as \eqref{diagonale} and set $\gamma_q(N)=(A_N A_N^*)_{qq}$.
 In this section, we let $X$ stand for ${\cal G}_N$, $A$ stands for $A_N$,  $G$ denotes the resolvent of $M_N= \Sigma \Sigma^*$ where $\Sigma=\sigma \frac{{\cal G}_N}{\sqrt{N}}+A_N$ and 
$g_N$ denotes the mean of the Stieltjes transform of the spectral measure of $M_N$, that is 
$$g_N(z) = \E\left(\frac{1}{n} Tr G(z)\right), \, z \in \C\setminus \R.$$

\subsection{Matricial master equation}
To obtain the equation (\ref{mean}) below, we will use many ideas from \cite{DHLN}. The following Gaussian integration by part formula is the key tool in our approach. 

\begin{lemme}\label{gaussian}[Lemma 2.4.5 \cite{AGZ09}]
Let $\xi$ be a real centered Gaussian random variable with variance $1$.  Let $ \Phi$ be a differentiable function  with polynomial growth of $\Phi$ and $\Phi'$. Then, 
$$\mathbb{E} \left( \xi \Phi(\xi) \right) = \mathbb{E} \left(  \Phi^{'}(\xi) \right).$$

\end{lemme}

\begin{proposition}\label{intbypart} Let $z$ be in $\mathbb{C}\setminus \mathbb{R}$. 
We have for any $(p,q)$ in $\{1,\ldots,n\}^2$,\\
~~

\noindent $\mathbb{E}\left( G_{pq}(z)\right) \left\{  z(1-\sigma^2 c_N g_N(z)) -\frac{\gamma_q(N)}{   1-\sigma^2 c_N g_N(z)} -\sigma^2 (1-c_N)                          +\frac{\sigma^2}{N}\sum_{p=1}^n \nabla_{pp}(z) \right\}$ \begin{equation}\label{mean} \hspace*{2cm}=\delta_{pq}
  +\nabla_{pq}(z),\end{equation}
where \begin{equation}\label{nablapq} \nabla_{pq} = \frac{1}{1- \sigma^2 c_N g_N}\left\{
 \frac{\sigma^2}{N} \frac{ \mathbb{E}\left( G_{pq}\right) }{1-\sigma^2 c_N g_N} \Delta_3 + \Delta_2(p,q) + \Delta_1(p,q).\right\},\end{equation}
\begin{equation}\label{delta1pq}\Delta_{1}(p,q) = {\sigma^2}\mathbb{E}\left\{ \left[\frac{1}{N} Tr G   - \mathbb{E}\left(\frac{1}{N} Tr G  \right)\right]    (G\Sigma \Sigma^*)_{pq}  \right\} ,\end{equation}
\begin{equation}\label{delta2pq}\Delta_{2}(p,q) = \frac{\sigma^2}{N}\mathbb{E}\left\{  Tr(GA\Sigma^* ) \left[G_{pq} - \mathbb{E}\left(G_{pq}\right) \right]  \right\} ,\end{equation}
\begin{equation}\label{delta3}\Delta_{3} = {\sigma^2}\mathbb{E}\left\{ \left[\frac{1}{N} Tr G   - \mathbb{E}\left(\frac{1}{N} Tr G \right)\right] Tr (\Sigma^*GA )   \right\} .\end{equation}
\end{proposition}

\begin{proof}
Using Lemma \ref{gaussian} with $\xi=\Re X_{ij} $ or $\xi= \Im X_{ij}$ and $\Phi= G_{pi} \overline{\Sigma_{qj}}$, we obtain that 
for any $j,q,p$,
\begin{eqnarray}
\mathbb{E}\left[ \left( G \frac{\sigma X}{\sqrt{N}}\right)_{pj} \overline{\Sigma_{qj}} \right] &=& \sum_{i=1}^n \mathbb{E}\left[ G_{pi} \frac{\sigma X_{ij} }{\sqrt{N}}\overline{\Sigma_{qj}} \right]\\&=&\frac{\sigma^2}{N} \sum_{i=1}^n  \mathbb{E}\left[ \left( G \Sigma\right)_{pj} G_{ii} \overline{\Sigma_{qj}} \right]
+ \frac{\sigma^2}{N} \mathbb{E}(G_{pq})\\&=& \frac{\sigma^2}{N}  \mathbb{E}\left[ \left( Tr G\right) \left(G \Sigma\right)_{pj}  \overline{\Sigma_{qj}} \right]
+ \frac{\sigma^2}{N} \mathbb{E}(G_{pq}). \label{1}
\end{eqnarray}
On the other hand, we have
\begin{eqnarray}
\mathbb{E}\left[ \left( G A \right)_{pj} \overline{\Sigma_{qj}} \right] &=&\mathbb{E}\left[ \left( G A \right)_{pj} \overline{A_{qj}} \right]
+\sum_{i=1}^n \mathbb{E}\left[  G_{pi} A_{ij}  \frac{ \sigma \overline{X_{qj}}}{\sqrt{N}} \right]\\&=& \mathbb{E}\left[ \left( G A \right)_{pj} \overline{A_{qj}} \right] +\frac{\sigma^2}{N} \mathbb{E}\left[  G_{pq} \left( \Sigma^* G A \right)_{jj} \right] \label{2}
\end{eqnarray}
where we applied Lemma \ref{gaussian} with $\xi=\Re X_{qj} $ or $\xi= \Im X_{qj}$ and $\Psi= G_{pi}A_{ij}$.
Summing (\ref{1}) and (\ref{2}) yields
\begin{eqnarray} \mathbb{E}\left[ \left( G \Sigma \right)_{pj} \overline{\Sigma_{qj}} \right] &=&\frac{\sigma^2}{N} \mathbb{E}(G_{pq})+
\frac{\sigma^2}{N}  \mathbb{E}\left[  \left(Tr G\right) \left(G \Sigma\right)_{pj}  \overline{\Sigma_{qj}} \right]\\&&+\frac{\sigma^2}{N} \mathbb{E}\left[  G_{pq} \left( \Sigma^* G A \right)_{jj} \right] + \mathbb{E}\left[ \left( G A \right)_{pj} \overline{A_{qj}} \right]. \label{delta1}\end{eqnarray}
Define
$$\Delta_1(j)= \frac{\sigma^2}{N}  \mathbb{E}\left[  \left(Tr G\right) \left(G \Sigma\right)_{pj}  \overline{\Sigma_{qj}} \right]- \frac{\sigma^2}{N}  \mathbb{E}\left[  Tr G \right] \mathbb{E}\left[ \left(G \Sigma\right)_{pj}  \overline{\Sigma_{qj}} \right].$$
From (\ref{delta1}), we can deduce that 
\begin{eqnarray*} \mathbb{E}\left[ \left( G \Sigma \right)_{pj} \overline{\Sigma_{qj}} \right] &=&\frac{1}{1-\sigma^2 c_N g_N} \left\{ 
\frac{\sigma^2}{N} \mathbb{E}(G_{pq})+
\frac{\sigma^2}{N} \mathbb{E}\left[  G_{pq} \left( \Sigma^* G A \right)_{jj} \right] \right.\\&&~~~~~~~~~~~~~~~~~~~~~~~~~~~~~\left.+ \mathbb{E}\left[ \left( G A \right)_{pj} \overline{A_{qj}} \right] + \Delta_1(j)\right\}.\end{eqnarray*}
Then, summing over $j$, we obtain that \\

\noindent 
$
 \mathbb{E}\left[ \left( G \Sigma \Sigma^*\right)_{pq} \right] =\frac{1}{1-\sigma^2 c_N g_N} \left\{ 
{\sigma^2} \mathbb{E}(G_{pq})+
\frac{\sigma^2}{N} \mathbb{E}\left[  G_{pq} Tr\left( \Sigma^* G A \right) \right] \right.$ \begin{equation}\label{71}\left.~~~~~~~~~~~~~~~~~~~~~~~~~~~~~~~~~~~~+ \mathbb{E}\left[ \left( G A A^*\right)_{pq}  \right] + \Delta_1 (p,q)\right\},\end{equation}
\noindent where $\Delta_{1}(p,q)$ is defined by (\ref{delta1pq}).
Applying Lemma \ref{gaussian} with $\xi=\Re X_{ij} $ or $ \Im X_{ij}$ and $\Psi= (GA)_{ij}$, we obtain that 
$$ \mathbb{E} \left[Tr\left( \frac{\sigma X^*}{\sqrt{N}} G A \right) \right] =\frac{\sigma^2}{N} \mathbb{E}\left[  Tr G ~Tr\left( \Sigma^* G A \right) \right]. $$
Thus, $$\mathbb{E}\left[ Tr\left( \Sigma^* G A \right) \right]=\mathbb{E}\left[ Tr\left( A^* G A \right) \right]+ \sigma^2 c_N g_N \mathbb{E}\left[ Tr\left( \Sigma^* G A \right) \right] + \Delta_3,$$
\noindent where $\Delta_3$ is defined by (\ref{delta3})
and then \begin{equation}\label{R3} \mathbb{E}\left[ Tr\left( \Sigma^* G A \right) \right]=\frac{1}{1-\sigma^2 c_Ng_N}  \left\{ \mathbb{E}\left[ Tr\left(  G A A^*\right) \right] + \Delta_3 \right\}. \end{equation}
(\ref{R3}) and (\ref{delta2pq}) imply that 
\begin{equation}\label{ok}\frac{\sigma^2}{N} \mathbb{E}\left[ G_{pq}Tr\left( \Sigma^* G A \right) \right]= \frac{\sigma^2}{N}\frac{\mathbb{E}(G_{pq})}{1-\sigma^2 c_N g_N}  \left\{ \mathbb{E}\left[ Tr\left(  G A A^*\right) \right] + \Delta_3 \right\} + \Delta_2(p,q), \end{equation}
\noindent where $\Delta_2(p,q)$ is defined by (\ref{delta2pq}).
We can deduce from (\ref{71}) and (\ref{ok}) that \\
~~

\noindent $ \mathbb{E}\left[ \left( G \Sigma \Sigma^*\right)_{pq} \right] $
$$=\frac{1}{1-\sigma^2 c_Ng_N} \left\{ 
{\sigma^2} \mathbb{E}(G_{pq})+  \mathbb{E}\left[ \left( G A A^*\right)_{pq}  \right] 
+\frac{\sigma^2}{N} \frac{\mathbb{E}\left[  G_{pq}\right] }{1-\sigma^2 c_N g_N}\mathbb{E}\left[   Tr\left( G AA^* \right) \right]\right.
$$
\begin{equation}\label{73} \left.\hspace*{4cm}+\frac{\sigma^2}{N}\frac{\mathbb{E}(G_{pq})}{1-\sigma^2 c_N g_N} \Delta_3
+ \Delta_1 (p,q)+ \Delta_2 (p,q) \right\}.\end{equation}
Using the resolvent identity and (\ref{73}), we obtain that 
\begin{eqnarray} z\mathbb{E}\left( G_{pq} \right)& =&\frac{1}{1-\sigma^2 c_Ng_N} \left\{ 
{\sigma^2} \mathbb{E}(G_{pq})+  \mathbb{E}\left[ \left( G A A^*\right)_{pq}  \right] \right.
 \nonumber \\
&& \left.~~~~~~~~~~~~+\frac{\sigma^2}{N} \frac{\mathbb{E}\left[  G_{pq}\right] }{1-\sigma^2 c_Ng_N} \mathbb{E}\left[   Tr\left( G AA^* \right)\right] \right\}+ \delta_{pq}+\nabla_{pq}. \label{115}\end{eqnarray}
where 
$\nabla_{pq}$ is defined by (\ref{nablapq}).
Taking $p=q$ in (\ref{115}), summing over $p$ and dividing by $n$, we obtain that 
\begin{eqnarray} 
z g_N &=& \frac{\sigma^2 g_N}{1-\sigma^2 c_N g_N} + \frac{ Tr \left[ \mathbb{E} (G) AA^*\right]}{n(1-\sigma^2 c_N g_N)} \\
&&+ \frac{\sigma^2 g_N Tr \left[ \mathbb{E} (G) AA^*\right]}{N(1-\sigma^2 c_N g_N)^2} +1 + \frac{1}{n} \sum_{p=1}^n \nabla_{pp}
\end{eqnarray}
It readily follows that 
\begin{equation}\label{terme}\frac{ Tr \left[ \mathbb{E} (G) AA^*\right]}{n(1-\sigma^2 c_Ng_N)} \left( \frac{ \sigma^2 c_N g_N}{(1-\sigma^2 c_N g_N)}  +1 \right)
=\left( z -  \frac{ \sigma^2 }{(1-\sigma^2 c_N  g_N)} \right) g_N -1 -  \frac{1}{n} \sum_{p=1}^n \nabla_{pp}.$$
Therefore $$\frac{ Tr \left[ \mathbb{E} (G) AA^*\right]}{n(1-\sigma^2 c_N  g_N)}=zg_N (1-\sigma^2 c_N  g_N) - (1-\sigma^2 c_N  g_N) -\sigma^2 g_N
- (1-\sigma^2 c_N g_N) \frac{1}{n} \sum_{p=1}^n \nabla_{pp}.\end{equation}
(\ref{terme}) and (\ref{115}) yield
$$ \mathbb{E}(G_{pq}) \times \left\{ z (1-\sigma^2 c_N g_N) -\frac{\gamma_q}{1-\sigma^2 c_N g_N} -\sigma^2(1-c_N)  +  \frac{\sigma^2}{N} \sum_{p=1}^n \nabla_{pp}\right\} =\delta_{pq} +  \nabla_{pq}.$$ 
Proposition \ref{intbypart} follows. \end{proof}
\subsection{Variance estimates}
In this section, when we state that some quantity $\Delta _N(z)$, $z \in \C\setminus \R$, 
is equal to $O(\frac{1}{N^p})$, this means precisely that there exist some polynomial $P$ with nonnegative coefficients and some positive real number $l$  which are all  independent of $N$ such that for any $z \in \C\setminus \R$, 
$$\vert \Delta _N(z)\vert \leq \frac{(\vert z\vert+1)^l P( | \Im z |^{-1}) }{N^p}.$$
We present now the different estimates on the variance. They rely on the following  Gaussian  Poincar\'e inequality (see the Appendix B). Let $Z_1,\ldots,Z_q$ be $q$ real  independent centered Gaussian variables with variance $\sigma^2$.  For any
${\cal C}^1$ function $f: \R^q \rightarrow \C$  such that $f$ and
${\rm grad} f$ are in $L^2({\cal N}(0, \sigma^2I_q))$, we have 
\begin{equation}\label{Poincare}
\mathbf{V}\left\{f(Z_1,\ldots ,Z_q)\right\}\leq \sigma^2 \mathbb{E} \left(\Vert ({\rm grad}f) (Z_1,\ldots,Z_q)\Vert_2^2\right) ,
\end{equation}
denoting for any random variable $a$ by  $\mathbf{V}(a) $ its variance $ \mathbb{E}(\vert a-\mathbb{E}(a)\vert^2)$.
Thus,  $(Z_1,\ldots,Z_q)$ satisfies a Poincar\'e inequality with constant $C_{PI}=\sigma^2$.

The following preliminary result will be useful to these estimates.

\begin{lemme}\label{bornelambda1}
There exists $K>0$ such for all $N$, $$\mathbb{E} \left( \lambda_1\left( \frac{XX^*}{N} \right) \right) \leq K.$$

\end{lemme}
\begin{proof} 
According to Lemma 7.2 in \cite{HT}, we have for any $t \in ]0;N/2]$,
$$\mathbb{E} \left[\Tr \left(\exp t  \frac{XX^*}{N} \right) \right]\leq n \exp \left( (\sqrt{c_N} +1)^2 t +\frac{1}{N} (c_N +1) t^2 \right).$$
By the Chebychev's inequality, we have 
\begin{eqnarray*}
\exp \left(t \mathbb{E} \left( \lambda_1\left( \frac{XX^*}{N} \right) \right) \right)& \leq& \mathbb{E}  \left( \exp t \lambda_1\left( \frac{XX^*}{N} \right)\right)\\ &\leq &\mathbb{E} \left[\Tr \left(\exp t  \frac{XX^*}{N} \right) \right]\\&\leq& n \exp \left( (\sqrt{c_N} +1)^2 t +\frac{1}{N} (c_N +1) t^2 \right).
\end{eqnarray*}
It follows that $$ \mathbb{E}\left(\lambda_1\left( \frac{XX^*}{N}\right)\right) \leq \frac{1}{t} \log n+ (\sqrt{c_N} +1)^2 + \frac{1}{N} (c_N +1) t.$$
The result follows by optimizing in $t$.
\end{proof}

\begin{lemme} \label{variance}
There exists $C>0$ such that for all large $N$, for all $z \in \mathbb{C}\setminus \mathbb{R}$, 
\begin{equation}\label{esttrace} \mathbb{E}\left( \left| \frac{1}{n}\Tr G - \mathbb{E}(\frac{1}{n}\Tr G)\right|^2 \right)\leq \frac{C}{N^2 \vert \Im z \vert^4},\end{equation}
\begin{equation}\label{Opq}\forall (p,q)\in \{1,\ldots,n\}^2, \;\mathbb{E}\left( \vert G_{pq} - \mathbb{E}(G_{pq})\vert^2 \right)\leq \frac{C}{N \vert \Im z \vert^4},\end{equation}
\begin{equation}\label{V}\mathbb{E}\left( \vert \Tr \Sigma^*GA  - \mathbb{E}(\Tr \Sigma^* GA)\vert^2 \right)\leq \frac{C(1+\vert z \vert)^2}{ \vert \Im z \vert^4}.\end{equation}
\end{lemme}
\begin{proof}
Let us define $\Psi: \mathbb{R}^{2(n\times N)} \rightarrow { M }_{n\times N}(\mathbb{C})$ by 
$$\Psi:~~\{x_{ij},y_{ij},i=1,\ldots,n,j=1,\ldots,N\}\rightarrow \sum_{i=1,\ldots,n}\sum_{j=1,\ldots,N} \left( x_{ij} +\sqrt{-1} y_{ij} \right) e_{ij},$$
where $e_{ij}$ stands for the $n\times N$ matrix such that for any $(p,q)$ in $\{1,\ldots,n\} \times \{1,\ldots, N\}$, $(e_{ij})_{pq}=\delta_{ip}\delta_{jq}.$
Let $F$ be a smooth complex function on ${ M }_{n\times N}(\mathbb{C})$ and 
define the complex function $f$ on  $\mathbb{R}^{2(n\times N)}$ by setting $f=F\circ \Psi$.
Then,$$\Vert {\rm grad} f(u)\Vert_2 = \sup_{V\in { M }_{n\times N}(\mathbb{C}), Tr VV^*=1}
\left| \frac{d}{dt} F(\Psi(u)+tV)_{\vert_{t=0}}\right|.$$
Now, $X=\Psi(\Re(X_{ij}),
\Im(X_{ij}),1\leq i\leq n,1\leq j\leq N)$ where the distribution of $\{\Re(X_{ij}),
\Im(X_{ij}),1\leq i\leq n,1\leq j\leq N\}$ is ${\cal N}(0, \frac{1}{2}I_{2nN})$.\\
Hence consider $F:~H \rightarrow \frac{1}{n} \Tr \left(zI_n -\left(\sigma\frac{ H}{\sqrt{N}} +A \right)\left(\sigma\frac{H}{\sqrt{N}} +A \right)^*\right)^{-1}$.\\
Let $ V\in {M }_{n\times N}(\mathbb{C})$ such that $ Tr VV^*=1$.\\

\noindent $\frac{d}{dt} F(X+tV)\vert_{t=0}$
$$=\frac{1}{n} \left\{\Tr \left(G\sigma\frac{V}{\sqrt{N}} \left(\sigma\frac{X}{\sqrt{N}} +A \right)^*  G\right) + \Tr\left(G \left(\sigma\frac{X}{\sqrt{N}} +A \right)\sigma \frac{V^*}{\sqrt{N}}  G\right)\right\}.$$
Moreover using Cauchy-Schwartz's inequality and Lemma \ref{lem0}, we have \\

\noindent $\left| \frac{1}{n} \Tr \left(G\sigma\frac{V}{\sqrt{N}} \left(\sigma\frac{X}{\sqrt{N}} +A \right)^*  G\right)\right|
$ \begin{eqnarray*} &\leq &\frac{\sigma}{n} (TrVV^* )^{\frac{1}{2}}\left[\frac{1}{N}Tr (\left(\sigma\frac{X}{\sqrt{N}} +A \right)\left(\sigma\frac{X}{\sqrt{N}} +A \right)^*G^2(G^*)^2)\right]^{\frac{1}{2}}\\
&\leq & \frac{\sigma}{\sqrt{N}\sqrt{n}\vert \Im z\vert^2}\left[\lambda_1\left(\left(\sigma\frac{X}{\sqrt{N}} +A \right)\left(\sigma\frac{X}{\sqrt{N}} +A \right)^*\right)\right]^{\frac{1}{2}}.
\end{eqnarray*}
We get obviously the same  bound for
$\vert \frac{1}{n} \Tr\left(G \left(\sigma\frac{X}{\sqrt{N}} +A \right) \sigma \frac{V^*}{\sqrt{N}}  G\right)\vert$. Thus \\

$\mathbb{E}\left(\Vert {\rm grad}f\left(\Re(X_{ij}),
\Im(X_{ij}),1\leq i\leq n,1\leq j\leq N\right)\Vert_2^2 \right)$ \begin{equation} \label{grad} \leq \frac{4 \sigma^2}{\vert \Im z\vert^4 Nn}\mathbb{E}\left[\lambda_1\left(\left(\sigma\frac{X}{\sqrt{N}} +A \right)\left(\sigma\frac{X}{\sqrt{N}} +A \right)^*\right)\right]. \end{equation}
(\ref{esttrace}) readily follows from (\ref{Poincare}), (\ref{grad}), Theorem A.8 in \cite{BaiSil06}, Lemma \ref{bornelambda1} and the fact that $\Vert A_N \Vert$ is uniformly bounded.
Similarly, considering $$F:~H \rightarrow \Tr \left[\left(zI_N -\left(\sigma \frac{H}{\sqrt{N}} +A \right)\left(\sigma \frac{H}{\sqrt{N}} +A \right)^*\right)^{-1}E_{qp}\right],$$
where $E_{qp}$ is the $n\times n$ matrix such that $(E_{qp})_{ij}=\delta_{qi}\delta_{pj}$, we can obtain that, for any $V\in { M }_{n\times N}(\mathbb{C})$ such that $ \Tr VV^*=1$,\\

\noindent  $\left| \frac{d}{dt} F(X+tV)_{\vert_{t=0}} \right|$ $$ \leq \frac{\sigma}{\sqrt{N}} \left\{\left(\left(GG^*\right)_{pp} \left(G^*\Sigma \Sigma^*G\right)_{qq}\right)^{1/2}
+ \left( \left(G^*G\right)_{qq} \left(G\Sigma \Sigma^*G^*\right)_{pp}\right)^{1/2}\right\}.$$Thus,  one can get \eqref{Opq} in the same way.
Finally, considering $$F:~H \rightarrow \Tr \left[\left(\sigma \frac{H}{\sqrt{N}} +A \right)^*\left(zI_N -\left(\sigma \frac{H}{\sqrt{N}} +A \right)\left(\sigma \frac{H}{\sqrt{N}} +A \right)^*\right)^{-1}A\right],$$ we can obtain that, for any $V\in { M }_{n\times N}(\mathbb{C})$ such that $ \Tr VV^*=1$,
 \begin{eqnarray*}\left| \frac{d}{dt} F(X+tV)_{\vert_{t=0} }\right|& \leq&\sigma\left\{
\left(\frac{1}{N} \Tr  \Sigma^* G A \Sigma^* GG^* \Sigma A^* G^*\Sigma \right)^{1/2} \right. \\&&\left.+ 
\left(\frac{1}{N} \Tr GA \Sigma^* G \Sigma \Sigma^* G^* \Sigma A^* G^* \right)^{1/2} \right. \\&&\left.+ 
\left(\frac{1}{N} \Tr GA  A^* G^* \right)^{1/2}\right\} \end{eqnarray*}

\noindent Using Lemma \ref{lem0} (i), Theorem A.8 in \cite{BaiSil06},
 Lemma \ref{bornelambda1},  the identity
$\Sigma \Sigma^*G=G\Sigma \Sigma^* = -I + z G,$ and the fact that $\Vert A_N \Vert$ is uniformly bounded, 
 the same analysis allows to prove \eqref{V}.
\end{proof}
\begin{corollaire}\label{estimeesvariance}
Let $\Delta_1(p,q)$, $\Delta_2(p,q)$, $(p,q)\in \{1,\ldots,n\}^2$, and $\Delta_3$ be as defined in Proposition \ref{intbypart}. Then there exist a polynomial $P$ with nonnegative coefficients and a nonnegative real number $l$ such that, for all large $N$, for any $z\in \mathbb{C}\setminus \mathbb{R}$,
\begin{equation}\label{estdelta3} \Delta_3(z)\leq \frac{P(\vert \Im z\vert^{-1} (1+\vert z \vert )^l}{N},\end{equation}
and  for all $ (p,q)\in \{1,\ldots,n\}^2$,  
\begin{equation}\label{estdelta1} \Delta_1(p,q)(z)\leq \frac{P(\vert \Im z\vert^{-1} (1+\vert z \vert )^l}{N},\end{equation}
\begin{equation}\label{estdelta2} \Delta_2(p,q)(z)\leq \frac{P(\vert \Im z\vert^{-1} (1+\vert z \vert )^l}{N\sqrt{N}}.\end{equation}

\end{corollaire}
\begin{proof}
Using the identity
$$GM_N = -I + z G,$$
(\ref{estdelta1}) readily follows from Cauchy-Schwartz inequality, Lemma \ref{lem0}  and (\ref{esttrace}). (\ref{estdelta2}) and (\ref{estdelta3})
readily follows from Cauchy-Schwartz inequality and Lemma \ref{variance}
\end{proof}
\subsection{Estimates of Resolvent entries}
In order to deduce Proposition  \ref{estimfonda}  from Proposition \ref{intbypart} and Corollary \ref{estimeesvariance}, we need the two following Lemma \ref{majpre}
and Lemma \ref{gnmoinsgmu}.

\begin{lemme}\label{majpre}
For all  $z\in \mathbb{C}\setminus \mathbb{R}$,
\begin{equation}\label{in1}\frac{1}{\left|1- \sigma^2 c_N g_N(z)\right|} \leq \frac{\vert z\vert }{\vert \Im z \vert},\end{equation}
\begin{equation}\label{in2}\frac{1}{\left|1- \sigma^2 c g_{\mu_{\sigma,\nu,c}}(z)\right|} \leq \frac{\vert z\vert }{\vert \Im z \vert}.\end{equation}

\end{lemme}
\begin{proof}

Since $\mu_{M_N}$ is supported by $[0,+\infty[$, \eqref{in1} readily follows from  \\

\noindent $\frac{1}{\left|1- \sigma^2 c_N g_N(z)\right|}  =\frac{\vert z\vert }{\left|z- \sigma^2 c_N zg_N(z)\right|} $ $$\hspace*{1,7
cm}\leq \frac{\vert z\vert }{\left|\Im (z- \sigma^2 c_N zg_N(z))\right|}= \frac{\vert z\vert }{|\Im  z |\left( 1+\sigma^2 c_N \mathbb{E} \int \frac{t}{|z-t|^2} d\mu_{M_N}(t)\right)}.$$
\eqref{in2} may be proved similarly.

\end{proof}
Corollary \ref{estimeesvariance} and 
Lemma \ref{majpre} yields that,  there is a polynomial $Q$ with nonnegative coefficients,  a sequence $b_N$ of nonnegative
real numbers converging to zero when $N$ goes to infinity 
and some nonnegative integer number $l$, 
such that for any  $p,q$ in $\{1, \ldots,n\} $, for all  $z\in \mathbb{C}\setminus \mathbb{R}$,
\begin{equation} \label{nablast}\nabla_{pq}  \leq (1+\vert z\vert)^l Q(\vert \Im z \vert^{-1})b_N,\end{equation}
where $\nabla_{pq}$ was defined by (\ref{nablapq}).\\

\begin{lemme}\label{gnmoinsgmu}
There is  a sequence $v_N$ of nonnegative
real numbers converging to zero when $N$ goes to infinity 
such that for all  $z\in \mathbb{C}\setminus \mathbb{R}$,

\begin{equation} \label{gnmoinsg} \left| g_N(z)-g_{\mu_{\sigma,\nu,c}}(z)\right|  \leq  \left\{\frac{\vert z\vert^2 +2}{\vert \Im z \vert^{2}}+ \frac{1}{\vert \Im z \vert}\right\}v_N.\end{equation}
\end{lemme}
\begin{proof}
First note that it is sufficient to prove (\ref{gnmoinsg}) for $z\in \mathbb{C}^+:=\{z \in \mathbb{C}; \Im z >0\}$ since 
$ g_N(\bar z)-g_{\mu_{\sigma,\nu,c}} (\bar z)= \overline{g_N(z)-g_{\mu_{\sigma,\nu,c}}(z)}$.
Fix $\epsilon>0.$
According to Theorem A.8 and  Theorem 5.11 in \cite{BaiSil06}, and the assumption on $A_N$, we can choose $K> \max\{ 2/\varepsilon; x, x \in \rm{supp}( \mu_{\sigma,\nu,c})\}$ large enough such that $\mathbb{P}\left( \left\|M_N\right\| >K\right)$
goes to zero as $N$ goes to infinity.
Let us write
\begin{equation}\label{reecriture}g_N(z)=\mathbb{E}\left( \frac{1}{n} \Tr G_N(z) {\1}_{\Vert M_N \Vert \leq K} \right) +\mathbb{E}\left( \frac{1}{n} \Tr G_N(z) {\1}_{\Vert M_N \Vert > K} \right).\end{equation}
For any $z \in \mathbb{C}^+$ such that $\vert z \vert > 2K$, we have
$$\left|\mathbb{E}\left( \frac{1}{n} \Tr G_N(z) {\1}_{\Vert M_N \Vert \leq K} \right)\right| \leq  \frac{1}{K}  \leq \frac{\epsilon}{2} \mbox{~~and~~}\left|g_{\mu_{\sigma,\nu,c}}(z)\right| \leq  \frac{1}{K}  \leq \frac{\epsilon}{2}. $$
Thus, $\forall z \in \mathbb{C}^+,$ such that $ \vert z \vert > 2K$, we can deduce that \\

$ \left|\mathbb{E}\left( \frac{1}{n} \Tr G_N(z) {\1}_{\Vert M_N \Vert \leq K} \right)-g_{\mu_{\sigma,\nu,c}}(z)\right|\frac{(\Im z)^2}{\vert z \vert^2 +2}$ \begin{eqnarray} &\leq & \left|\mathbb{E}\left( \frac{1}{n} \Tr G_N(z) {\1}_{\Vert M_N \Vert \leq K} \right)-g_{\mu_{\sigma,\nu,c}} (z)\right| \nonumber \\&\leq& {\varepsilon} \label{horscompact}.\end{eqnarray}
Now, it is clear that $\mathbb{E}\left( \frac{1}{n} \Tr G_N {\1}_{\Vert M_N \Vert \leq K} \right)$ is a sequence of locally bounded holomorphic functions on $\mathbb{C}^+$ which converges towards $g_{\mu_{\sigma,\nu,c}}$. Hence, by Vitali's Theorem, $\mathbb{E}\left( \frac{1}{n} \Tr G_N {\1}_{\Vert M_N \Vert \leq K} \right)$ converges uniformly towards $g_{\mu_{\sigma,\nu,c}}$ on each compact subset of $\mathbb{C}^+$.
Thus, there exists $N(\epsilon)>0$, such that for any $N \geq N(\epsilon)$, for any $z\in \mathbb{C}^+$,  such that $ \vert z \vert \leq 2K$ and $\Im z \geq {\varepsilon}$,\\

\noindent $
 \left|\mathbb{E}\left( \frac{1}{n} \Tr G_N(z) {\1}_{\Vert M_N \Vert \leq K} \right)-g_{\mu_{\sigma,\nu,c}}(z)\right|\frac{(\Im z)^2}{\vert z \vert^2 +2}$
\begin{eqnarray}&\leq & \left|\mathbb{E}\left( \frac{1}{n} \Tr G_N(z) {\1}_{\Vert M_N \Vert \leq K} \right)-g_{\mu_{\sigma,\nu,c}}(z)\right| \nonumber \\&\leq&{ \varepsilon} \label{compactaudessus}.\end{eqnarray}
Finally, for any $z\in \mathbb{C}^+$,  such that  $\Im z \in ]0; {\varepsilon}[$, we have 
\begin{equation}\label{compactaudessous}
 \left|\mathbb{E}\left( \frac{1}{n} \Tr G_N(z) {\1}_{\Vert M_N \Vert \leq K} \right)-g_{\mu_{\sigma,\nu,c}}(z)\right|\frac{(\Im z)^2}{\vert z \vert^2 +2} \leq \frac{2}{\Im z} \frac{(\Im z)^2}{\vert z \vert^2 +2} \leq \Im z \leq {\varepsilon}.\end{equation}
It readily follows from  (\ref{horscompact}), (\ref{compactaudessus}) and (\ref{compactaudessous}) that for $N \geq N(\epsilon)$, $$ \sup_{z\in \mathbb{C}^+}\left\{ \left|\mathbb{E}\left( \frac{1}{n} \Tr G_N(z) {\1}_{\Vert M_N \Vert \leq K} \right)-g_{\mu_{\sigma,\nu,c}}(z)\right|\frac{(\Im z)^2}{\vert z \vert^2 +2} \right\}\leq {\varepsilon}$$
Moreover, for $N \geq  N'(\epsilon)\geq N(\epsilon)$, $\mathbb{P}\left( \left\|M_N\right\| >K\right) \leq \varepsilon.$ Therefore, for $N \geq  N'(\epsilon)$, we have for any $z \in \mathbb{C}^+$,\\
~~

\noindent $\left| g_N(z)-g_{\mu_{\sigma,\nu,c}}(z)\right|$
\begin{eqnarray}& \leq &\frac{\vert z\vert^2 +2}{\vert \Im z \vert^{2}} \sup_{z\in \mathbb{C}^+} \left\{\left|\mathbb{E}\left( \frac{1}{n} \Tr G_N(z) {\1}_{\Vert M_N \Vert \leq K} \right)-g_{\mu_{\sigma,\nu,c}}(z)\right|\frac{(\Im z)^2}{\vert z \vert^2 +2}\right\} \nonumber\\&& + \frac{1}{\Im z} \mathbb{P}\left( \left\|M_N\right\| >K\right)\nonumber\\&\leq & 
\varepsilon  \left\{\frac{\vert z\vert^2 +2}{\vert \Im z \vert^{2}}+ \frac{1}{\Im z}\right\}
\end{eqnarray}
Thus,  the proof is complete by setting 
$$v_N= \sup_{z\in \mathbb{C}^+} \left\{\left|g_N(z)-g_{\mu_{\sigma,\nu,c}}(z) \right|  \left(\frac{\vert z\vert^2 +2}{\vert \Im z \vert^{2}}+ \frac{1}{\Im z}\right)^{-1}\right\}.$$
\end{proof}

\noindent Now  set
$$
\tau_N= {(1-\sigma^2c_Ng_{ N}(z))z- \frac{ \gamma_q(N)}{1- \sigma^2 c_Ng_{N}(z)} -\sigma^2 (1-c_N)}$$ and \begin{equation} \label{tau'} \tilde \tau_N={(1-\sigma^2cg_{ \mu_{\sigma,\nu,c}}(z))z- \frac{ \gamma_q(N)}{1- \sigma^2 cg_{ \mu_{\sigma,\nu,c}}(z)} -\sigma^2 (1-c)}.\end{equation}
 Lemmas \ref{majpre} and \ref{gnmoinsgmu} yield that there is a polynomial $R$ with nonnegative coefficients,  a sequence $w_N$ of nonnegative
real numbers converging to zero when $N$ goes to infinity 
and some nonnegative real number $l$, 
such that for all  $z\in \mathbb{C}\setminus \mathbb{R}$, 
\begin{equation}\label{tau}\left|\tau_N - \tilde  \tau_N\right| \leq (1+\vert z\vert)^l R(\vert \Im z \vert^{-1})w_N.\end{equation}
Now, one can easily see  that, 
  \begin{equation} \label{ima2} \left|\Im \left\{(1-\sigma^2cg_{ \mu_{\sigma,\nu,c}}(z))z- \frac{ \gamma_q(N)}{1- \sigma^2 cg_{ \mu_{\sigma,\nu,c}}(z)} -\sigma^2 (1-c)\right\}\right| \geq \vert  \Im z \vert,\end{equation}
so that 
\begin{equation}\label{moinsun} \left| \frac{1}{\tilde \tau_N}\right| \leq \frac{1}{\vert\Im z \vert}. \end{equation}
Note  that
\begin{equation}\label{re}\frac{1}{ \tilde \tau_N} =\frac{( {1- \sigma^2c  g_{\mu_{\sigma,\nu,c}}(z)})}{\omega_{\sigma, \nu, c}(z) -\gamma_q(N)}.\end{equation}

Then, (\ref{premier}) readily follows from Proposition \ref{intbypart}, (\ref{nablast}), (\ref{tau}), (\ref{moinsun}),  (\ref{re}),  and (ii) Lemma \ref{lem0}.
 The proof of  Proposition \ref{estimfonda} is complete.
\section{Proof of Theorem \ref{cvev1p1}}
We follow the two steps presented in Section 2.\\
{\bf Step A.}  We first prove (\ref{5.4}).

\noindent Let $\eta>0$ small enough  and $N$ large enough such that for any $l=1,\ldots, J$, $\alpha_l(N)\in [\theta_l-\eta,\theta_l+\eta]$ 
and $[\theta_l-2\eta,\theta_l+2\eta]$ contains no other element of the spectrum of $A_NA_N^*$  than $\alpha_l(N)$.
For any $l=1,\ldots,J$, choose $f_{\eta,l}$ in $\mathcal{C}^\infty (\mathbb{R}, \mathbb{R})$ with support in $[\theta_l 
-2\eta,\theta_l+2\eta]$ such that $f_{\eta,l}(x)=1$ for any $x \in [\theta_l 
-\eta,\theta_l+\eta]$ and $0 \leq f_{\eta,l} \leq 1$.
Let $0< \epsilon <\delta_0$ where $\delta_0$ is introduced in Theorem \ref{ThmASCV}. Choose $h_{\varepsilon,j}$ in $\mathcal{ C}^\infty (\mathbb{R}, \mathbb{R})$ with 
support in $[\rho_{\theta_j} -\varepsilon
,\rho_{\theta_j}+\varepsilon
]$  such that $h_{\varepsilon,j} \equiv 1$ on $[\rho_{\theta_j} -\varepsilon/2
,\rho_{\theta_j}+\varepsilon/2
]$ and $0 \leq h_{\varepsilon
,j}\leq 1$.\\
Almost surely for all large $N$, $M_N$ has $k_j$ eigenvalues in $]\rho_{\theta_j} -\varepsilon/2
,\rho_{\theta_j}+\varepsilon/2[$.
 According to Theorem \ref{ThmASCV}, denoting by   $(\xi_1,\cdots,\xi_{k_j})$  an  orthonormal system of eigenvectors associated to the $k_j$ eigenvalues of $M_N$ in $( \rho_{\theta_j} -\varepsilon/2, \rho_{\theta_j}+\varepsilon/2)$, it readily follows from (\ref{equavec}) that almost surely for all large $N$, 
$$\sum_{n=1}^{k_j}\left\| P_{\ker(\alpha_l(N) I_n-A_NA_N^*)}\xi_n \right\|^2= {\rm Tr} \left[ h_{\varepsilon
,j}(M_N) f_{\eta,l}(A_NA_N^*)\right].$$
Applying Proposition \ref{compgaunogau} with $\Gamma_N= f_{\eta,l}(A_NA_N^*)$ and $K=k_l$, the problem of establishing (\ref{5.4}) is reduced to prove that \\

$ \mathbb{E}\left(\Tr \left[h_{\varepsilon,j} \left(\left(\sigma\frac{{\cal G}_N}{\sqrt{N}}+A_N\right)\left(\sigma\frac{{\cal G}_N}{\sqrt{N}}+A_N\right)^*\right)  f_{\eta,l}(A_NA_N^*)\right] \right)$ \begin{equation}~~~~~~~~~~~~~~~~~~~~~~~~~~~~~~~~\rightarrow_{N \rightarrow +\infty} \frac{k_j\delta_{jl} (1-\sigma^2 c g_{\mu_{\sigma,\nu,c}}(\rho_{\theta_j}))}{\omega_{\sigma,\nu,c}'(\rho_{\theta_j})}. 
\end{equation}
Using a Singular Value Decomposition of $A_N$ and the biunitarily invariance of the distribution of ${\cal G}_N$, we can assume that 
$A_N$ is as \eqref{diagonale} and such that 
 for any $j=1,\ldots, J,$ $$(A_NA_N^*)_{ii}=\alpha_j(N) \mbox{~~ for $i=k_1+\ldots+k_{j-1}+l$, $l=1,\ldots,k_j$}.$$
Now, according to Lemma \ref{approxpoisson}, \\

$ \mathbb{E}\left(\Tr \left[h_{\varepsilon,j} \left(\left(\sigma \frac{{\cal G}_N}{\sqrt{N}}+A_N\right)\left(\sigma \frac{{\cal G}_N}{\sqrt{N}}+A_N\right)^*\right)  f_{\eta,l}(A_NA_N^*)\right] \right)$ $$= - \lim_{y\rightarrow 0^{+}}\frac{1}{\pi} \int \Im \mathbb{E}\Tr \left[G^{\cal G}_N(t+iy) f_{\eta,l}(A_NA_N^*)\right] h_{\varepsilon,j}(t) dt,$$
with,
for all large $N$, \begin{eqnarray*}\mathbb{E}\Tr \left[G^{\cal G}_N(t+iy) f_{\eta,l}(A_NA_N^*)\right]
&=&\sum_{k=k_1+\cdot+k_{l-1}+1}^{k_1+\cdot+k_{l}}
f_{\eta,l} (\alpha_l(N))\mathbb{E}[G^{\cal G}_N(t+iy)]_{kk} \\&=& \sum_{k=k_1+\cdot+k_{l-1}+1}^{k_1+\cdot+k_{l}}
\mathbb{E}[G^{\cal G}_N(t+iy)]_{kk}. \end{eqnarray*}
Now, 
 by considering $$\tau'={(1-\sigma^2cg_{ \mu_{\sigma,\nu,c}}(z))z- \frac{ \theta_l}{1- \sigma^2 cg_{ \mu_{\sigma,\nu,c}}(z)} -\sigma^2 (1-c)}$$ instead of dealing with $\tilde \tau_N$ defined in (\ref{tau'}) at the end of  the proof of Proposition \ref{estimfonda}, one can prove that  
 there is a polynomial $P$ with nonnegative coefficients,  a sequence $(u_N)_N$ of nonnegative
real numbers converging to zero when $N$ goes to infinity 
and some nonnegative real number $s$, 
such that for any $k$ in $\{k_1+\ldots+k_{l-1}+1, \ldots,k_1+\ldots+k_l\}$, for all  $z\in \mathbb{C}\setminus \mathbb{R}$,
\begin{equation} \label{deuxieme}
\mathbb{E} \left(\left( G^{\cal G}_N(z)\right)_{kk}\right) = \frac{1- \sigma^2 cg_{\mu_{\sigma,\nu,c}}(z)}{\omega_{\sigma, \nu, c}(z) -\theta_l} 
+\Delta_{k,N}(z),
\end{equation}
with $$\left| \Delta_{k,N} (z)\right| \leq (1+\vert z\vert)^s P(\vert \Im z \vert^{-1})u_N.$$
Thus, 
$$\mathbb{E}\Tr \left[G^{\cal G}_N(t+iy) f_{\eta,l}(A_NA_N^*)\right]
=  k_l \frac{1- \sigma^2 cg_{\mu,\sigma,\nu}(t+iy)}{\omega_{\sigma, \nu, c}(z) -\theta_l}
 + \Delta_N(t+iy), $$
where for all  $z \in \mathbb{C} \setminus  \mathbb{R}$,  $\Delta_N(z)= \sum_{k=k_1+\cdot+k_{l-1}+1}^{k_1+\cdot+k_{l}} \Delta_{k,N}(z),$
and   $\left| \Delta_{N} (z)\right| \leq k_l (1+\vert z\vert)^s P(\vert \Im z \vert^{-1})u_N.$

\noindent First let us compute
$$
\lim_{y\downarrow0}\frac{k_l}{\pi}\int_{\rho_{\theta_j}-\varepsilon}^{\rho_{\theta_j}+\varepsilon}
\Im\frac{h_{\varepsilon,j}(t) (1-\sigma^2 c g_{\mu_{\sigma,\nu,c}}(t+iy))}{\theta_l-\omega_{\sigma,\nu,c}(t+iy)}\,dt.
$$
The function $\omega_{\sigma,\nu,c}$ satisfies $\omega_{\sigma,\nu,c}(\overline{z})=\overline{\omega_{\sigma,\nu,c}(z)}$ and  $g_{\mu_{\sigma,\nu,c}}(\overline{z})=\overline{g_{\mu_{\sigma,\nu,c}}(z)}$, so that
$\Im\frac{ (1-\sigma^2 c g_{\mu_{\sigma,\nu,c}}(t+iy))}{\theta_l-\omega_{\sigma,\nu,c}(t+iy)}=\frac{1}{2i}[\frac{ (1-\sigma^2 c g_{\mu_{\sigma,\nu,c}}(t+iy))}{\theta_l-\omega_{\sigma,\nu,c}(t+iy)}-
\frac{ (1-\sigma^2 c g_{\mu_{\sigma,\nu,c}}(t-iy))}{\theta_l-\omega_{\sigma,\nu,c}(t-iy)}]$. As in \cite{MCJTP}, the above integral is split 
into three pieces, namely $\int_{\rho_{\theta_j}-\varepsilon}^{\rho_{\theta_j}-\varepsilon/2}+
\int_{\rho_{\theta_j}-\varepsilon/2}^{\rho_{\theta_j}+\varepsilon/2}+\int_{\rho_{\theta_j}+\varepsilon/2}^{\rho_{\theta_j}+\varepsilon}$. Each
of the first and third integrals are easily seen to go to zero when $y\downarrow0$ by a direct 
application of the definition of the functions involved and of the (Riemann) integral.
As $h_{\varepsilon,j}$ is constantly equal to one on $[\rho_{\theta_j}-\epsilon/2; \rho_{\theta_j}+\epsilon/2]$, the second (middle) term is simply the integral 
$$
\frac{k_l}{2\pi i}\int_{\rho_{\theta_j}-\varepsilon/2}^{\rho_{\theta_j}+\varepsilon/2}\frac{1-\sigma^2cg_{\mu_{\sigma,\nu,c}}(t+iy)}{\theta_l-\omega_{\sigma,\nu,c}(t+iy)}-
\frac{1-\sigma^2cg_{\mu_{\sigma,\nu,c}}(t-iy)}{\theta_l-\omega_{\sigma,\nu,c}(t-iy)}\,dt.
$$
Completing this to a contour integral on the rectangular with corners $\rho_{\theta_j}\pm\varepsilon/2\pm iy$
and noting that the integrals along the vertical lines tend to zero as $y\downarrow0$
allows a direct application of the residue theorem for the final result, if $l=j$, 
$$
\frac{k_j (1-\sigma^2 c g_{\mu_{\sigma,\nu,c}}(\rho_{\theta_j}))}{\omega_{\sigma,\nu,c}'(\rho_{\theta_j})}.
$$
 If we consider $\theta_l$ for some $l\neq j$, then $z\mapsto (1-\sigma^2 cg_{\mu_{\sigma,\nu,c}}(z))(\theta_l-
\omega_{\sigma,\nu,c}(z))^{-1}$ is analytic around $\rho_{\theta_j}$, so its residue at $\rho_{\theta_j}$ is zero, and the above 
argument provides zero as answer.

\noindent Now, according to Lemma \ref{HT}, we  have
$$
\limsup_{y\rightarrow 0^+}~(u_N)^{-1}\left|\int h_{\varepsilon,j}
(t)\Delta_N(t+iy)dt\right| <+\infty
$$
so that 
\begin{equation}\label{reste}
\lim_{N\rightarrow + \infty}\limsup_{y \rightarrow 0^+} \left| \int h_{\varepsilon,j}
(t) \Delta_N(t+iy)dt 
\right| =0.
\end{equation}

\noindent This concludes the proof of (\ref{5.4}).\\
~~

\noindent{\bf Step B:} In the second, and final, step, we shall use a perturbation
argument identical to the one used in \cite{MCJTP} to reduce the problem to the case of a 
spike with multiplicity one, case that follows trivially from Step A.
A further property of eigenvectors of Hermitian matrices which are close to each other in the norm 
will be important in the analysis of the behaviour of the eigenvectors of our matrix models. 
Given a Hermitian  matrix $M\in\ M_N(\mathbb C)$ and a Borel set $S\subseteq\mathbb R$, we
denote by $E_M(S)$ the spectral projection of $M$ associated to $S$. In other words, the range of 
$E_M(S)$ is the vector space generated by the eigenvectors of $M$ corresponding to eigenvalues 
in $S$. The 
following lemma can be found in \cite{BBCF15}. 
\begin{lemme}\label{eigenspaces}
Let $M$ and $M_0$ be $N \times N$  Hermitian matrices.  Assume that $\alpha,\beta,\delta\in\mathbb R$ are such that
$\alpha<\beta$, $\delta>0$, $M$  and $M_0$ has no eigenvalues in $[\alpha-\delta,\alpha]\cup
[\beta,\beta+\delta]$. Then, $$
\|E_{M}((\alpha,\beta))-E_{M_0}((\alpha,\beta))\|<\frac{4(\beta-\alpha+2\delta)}{\pi\delta^2}\|M-M_0\|.
$$
In particular, for any unit vector $\xi\in E_{M_0}((\alpha,\beta))(\mathbb C^N)$,
$$
\|(I_N-E_{M}((\alpha,\beta)))\xi\|_2<\frac{4(\beta-\alpha+2\delta)}{\pi\delta^2}\|M-M_0\|.
$$
\end{lemme}
\noindent Assume that $\theta_i$ is  in $\Theta_{\sigma,\nu,c}$ defined in (\ref{defTheta}) and $k_i\neq 1$.
Let us denote by   $V_1(i),\ldots, V_{k_i}(i)$, an orthonormal system of eigenvectors of $A_NA_N^*$ associated with $\alpha_i(N)$. 
Consider a Singular Value Decomposition $A_N=U_ND_NV_N$
where $V_N$ is a $N\times N$ unitary matrix, $U_N$ is a  $n\times n$ unitary matrix whose $k_i$ first columns are $ V_1(i),\ldots, V_{k_i}(i)$ and 
$D_N$ is as \eqref{diagonale} with the first $k_i$  diagonal elements equal to $\sqrt{\alpha_i(N)}$.

Let $ \delta_0$ be as in  Theorem \ref{ThmASCV}. Almost surely, for all $N$
large enough, there are $k_i$ eigenvalues of $M_N$ in $(\rho_{\theta_i}- \frac{\delta_0}{4}, \rho_{\theta_i}+ \frac{\delta_0}{4})$, namely  $\lambda_{n_{i-1}+q}(M_N)$, $q=1,\ldots,k_i$ (where $n_{i-1}+1,\ldots,n_{i-1}+k_i $ are the descending ranks of $\alpha_i(N)$ among the eigenvalues of $A_NA_N^*$),  which are moreover 
the only eigenvalues of $M_N$ in $(\rho_{\theta_i}-\delta_0,\rho_{\theta_i}+\delta_0)$. 
 Thus,  the spectrum of $M_N$ is  split into three pieces: $$\{\lambda_1(M_N),\dots,\lambda_{n_{i-1}}(M_N)\}\subset (\rho_{\theta_i}+\delta_0,+\infty[,$$ $$\{\lambda_{n_{i-1}+1}(M_N),\dots,\lambda_{n_{i-1}+k_i}(M_N)\}
\subset(\rho_{\theta_i}- \frac{\delta_0}{4},\rho_{\theta_i}+ \frac{\delta_0}{4}), $$ $$\{\lambda_{n_{i-1}+k_i+1}(M_N),\dots,
\lambda_{N}(M_N)\}\newline\subset [0,
\rho_{\theta_i}-\delta_0).$$ The distance between any of these
components is equal to $3\delta_0/4$.
 Let us fix  $\epsilon_0$ such that $0\leq \theta_i ( 2 \epsilon_0 k_i  +\epsilon_0^2k_i^2) <  dist(\theta_i, \mbox{supp~}\nu \cup_{i\neq s}\theta_s )$
and such that $[\theta_i;  \theta_i+\theta_i ( 2 \epsilon_0 k_i  +\epsilon_0^2k_i^2)] \subset {\cal E}_{\sigma, \nu, c}$ defined by (\ref{cale}).
For any $0< \epsilon<\epsilon_0$, define the matrix $A_N(\epsilon)$ as $A_N(\epsilon)=U_ND_N(\epsilon) V_N$ where 
$$\left(D_N(\epsilon)\right)_{m, m}=\sqrt{\alpha_{i}(N)} [1 + \epsilon (k_i-m+1)],\text{~~for $m\in\{1,\ldots,k_i\}$},$$
and  $\left(D_N(\epsilon)\right)_{pq}=\left(D_N\right)_{pq}$ for any $(p,q)\notin \{  (m,m), m\in\{1,\ldots,k_i\}\}$.

\noindent Set $$M_N(\epsilon)=\left(\sigma \frac{X_N}{\sqrt{N}} +A_N(\epsilon)\right)\left( \sigma \frac{X_N}{\sqrt{N}} +A_N(\epsilon)\right)^*.$$
For $N $ large enough, for each $m\in\{1,\ldots,k_i\}$, $ \alpha_{i}(N) [1 + \epsilon (k_i-m+1)]^2$  is an eigenvalue   of $A_NA_N^*(\epsilon)$ with multiplicity one. Note that, since $\sup_N\Vert A_N \Vert <+\infty$, it is easy to see that there exist some constant $C$ such that for any $N$ and  for any $0< \epsilon<\epsilon_0$,
$$
\left\|M_N(\epsilon)-M_N\right\|\leq
C\epsilon \left( \left\| \frac{X_N}{\sqrt{N}}\right\| +1\right) .
$$
Applying Remark \ref{2.1} to  the $(n+N)\times (n+N)$ matrix  $\tilde X_N=
\left( \begin{array}{ll} 0_{n\times n}~~~ X_N\\ X_N^*~~~ 0_{N\times N} \end{array} \right)$ (see also Appendix B of \cite{CSBD}), it readily follows that 
 there exists some constant $C'$ such that a.s for all large N,  for any $0< \epsilon<\epsilon_0$,
\begin{equation}\label{normediff} \left\| M_N(\epsilon)-M_N\right\| \leq C' \epsilon.\end{equation}
Therefore, for  $\epsilon $ sufficiently small  such that $C' \epsilon < \delta_0/4$, by Theorem A.46 \cite{BaiSil06}, there
are precisely $n_{i-1}$ eigenvalues of $M_N(\epsilon)$ in $[0,\rho_{\theta_i}-3\delta_0/4)$, precisely $k_i$ in $(\rho_{\theta_i}-\delta_0/2,\rho_{\theta_i}+\delta_0/2)$ and 
precisely $N-(n_{i-1}+k_i)$ in $(\rho_{\theta_i}+3\delta_0/4,+ \infty[$. All these intervals are again at strictly positive distance from
each other, in this case  $\delta_0/4$.

Let $\xi$ be a normalized  eigenvector of $M_N$ relative to $\lambda_{n_{i-1}+q}(M_N)$ for some 
$q\in\{1,\ldots, k_i\}$. As proved in Lemma \ref{eigenspaces}, if $E(\epsilon)$ denotes the
subspace spanned by the eigenvectors associated to $\{\lambda_{n_{i-1}+1}(M_N(\epsilon)),\dots,\lambda_{n_{i-1}+k_i}
(M_N(\epsilon))\}$ in $\mathbb C^N$, then there exists some constant $C$ (which depends on $\delta_0$) such that for $\epsilon$ small enough, almost surely for large $N$, 
\begin{equation}\label{least2}
\left\|P_{ E(\epsilon)^{\bot}}\xi\right\|_2\leq C \epsilon
.\end{equation}
According to  Theorem  \ref{ThmASCV}, 
for $j\in\{1,\ldots,k_i\}$, for large enough $N$, $\lambda_{n_{i-1}+j}(M_N(\epsilon
))$  
separates from the rest of the spectrum and belongs to a neighborhood of $\Phi_{\sigma,\nu,c} (\theta_i^{(j)}(\epsilon))$
where
$$\theta_i^{(j)}(\epsilon)=\theta_i\left( 1+\epsilon (k_i -j+1) \right)^2.$$
 If $\xi_{j}(\epsilon,i
)$ denotes a normalized eigenvector associated to $\lambda_{n_{i-1}+j}(M_N(\epsilon
))$, Step A above
implies that almost surely for any  $p\in \{1,\ldots, k_i\}$, for any $\gamma>0$, for all large $N$,
\begin{equation}\label{5.14}
\left|\left| \langle V_{p}(i),\xi_{j}(\epsilon
,i)\rangle \right|^2 -  \frac{\delta_{jp}\left(1-\sigma^2 c g_{\mu_{\sigma,\nu,c}}(\Phi_{\sigma,\nu,c} (\theta_i^{(j)}(\epsilon)))\right)}{ \omega_{\sigma,\nu,c}'\left(\Phi_{\sigma,\nu,c} (\theta_i^{(j)}(\epsilon))
\right)}\right|<\gamma.
\end{equation}
The eigenvector $\xi$ decomposes uniquely in the orthonormal basis of eigenvectors of $M_N(
\epsilon)$ as $\xi=\sum_{j=1}^{k_i}c_j(\epsilon)\xi_j(\epsilon,i)+\xi(\epsilon)^\perp$,
where $c_j(\epsilon)=\langle\xi|\xi_j(\epsilon,i)\rangle$ and $\xi(\epsilon)^\perp=P_{ E(\epsilon)^{\bot}}\xi$; necessarily 
$\sum_{j=1}^{k_i}|c_j(\epsilon)|^2+\|\xi(\epsilon)^\perp\|_2^2=1$. Moreover, as
indicated in relation \eqref{least2}, $\|\xi(\epsilon)^\perp\|_2\leq C \epsilon.$
We have
\begin{eqnarray*}
P_{\ker(\alpha_i(N)I_N-A_NA_N^*)}\xi&=&\sum_{j=1}^{k_i}c_j(\epsilon)
P_{\ker(\alpha_i(N)I_N-A_NA_N^*)}\xi_j(\epsilon,i)\\&&+P_{\ker(\alpha_i(N)I_N-A_NA_N^*)}\xi(\epsilon)^\perp\\
&=&\sum_{j=1}^{k_i}c_j(\epsilon)
\sum_{l=1}^{k_i}\langle \xi_j(\epsilon,i) | V_{l}(i)\rangle 
V_{l}(i)\\
& & \mbox{}+P_{\ker(\alpha_i(N)I_N-A_NA_N^*)}\xi(\epsilon)^\perp.
\end{eqnarray*}
Take in the above the scalar product with $\xi=\sum_{j=1}^{k_i}c_j(\epsilon)\xi_j(\epsilon,i)+\xi(\epsilon)^\perp$ to get
\begin{eqnarray*}
\lefteqn{\langle P_{\ker(\alpha_i(N)I_N-A_NA_N^*)}\xi|\xi\rangle =}\\
& & \mbox{}
\sum_{j,l,s=1}^{k_i}c_j(\epsilon)\langle  \xi_j(\epsilon,
i)     |V_{l}(i) \rangle\overline{c_s(\epsilon)}\langle V_{l}(i)|\xi_s(\epsilon,
i)\rangle\\
& & \mbox{}+\sum_{j=1}^{k_i}c_j(\epsilon)
\sum_{l=1}^{k_i}\langle  \xi_j(\epsilon,i)| V_{l}(i)     \rangle 
\langle V_{l}(i)|\xi(\epsilon)^\perp\rangle\\
& & \mbox{}+\langle P_{\ker(\alpha_i(N)I_N-A_NA_N^*)}\xi(\epsilon)^\perp|\xi\rangle.
\end{eqnarray*}
Relation \eqref{5.14} indicates that 
\begin{eqnarray*} \lefteqn{ \hspace*{-2cm}\sum_{j,l,s=1}^{k_i}c_j(\epsilon)\langle      \xi_j(\epsilon,
i)   |    V_{l}(i)          \rangle\overline{c_s(\epsilon)}\langle V_{l}(i)|\xi_s(\epsilon,
i)\rangle}\\
&= & \mbox{} \sum_{j=1}^{k_i}|c_j(\epsilon)|^2|\langle V_{j}(i)|\xi_j(\epsilon,
i)\rangle|^2+\Delta_1\\
 & =& \mbox{} \sum_{j=1}^{k_i}|c_j(\epsilon)|^2  \frac{\left(1-\sigma^2 c g_{\mu_{\sigma,\nu,c}}(\Phi_{\sigma,\nu,c} (\theta_i^{(j)}(\epsilon)))\right)}{ \omega_{\sigma,\nu,c}'\left(\Phi_{\sigma,\nu,c} (\theta_i^{(j)}(\epsilon))
\right)}+\Delta_1 + \Delta_2,
\end{eqnarray*}
where for all large $N$, $\vert \Delta_1\vert \leq \sqrt{ \gamma} k_i^3$ and $\vert \Delta_2\vert \leq \gamma $. Since 
$\|\xi(\epsilon)^\perp\|_2\leq C\epsilon$, \\

\noindent 
$
\left|\sum_{j=1}^{k_i}c_j(\epsilon)
\sum_{l=1}^{k_i}\langle \xi_j(\epsilon,i) | V_{l}(i)     \rangle 
\langle V_{l}(i)|\xi(\epsilon)^\perp\rangle \right.$\\

\noindent 
$\left.~~~~~~~~~~~~~~~~~~~~~~~~~~~~~~~~~~~~~~~+\langle P_{\ker(\alpha_i(N)I_N-A_NA_N^*)}\xi(\epsilon)^\perp|\xi\rangle\right|
\leq\left(k_i^2+1\right){C\epsilon}.
$\\

\noindent Thus, we conclude that almost surely for any $\gamma>0$, for all large $N$,$$
\left|\langle P_{\ker(\alpha_i(N)I_N-A_NA_N^*)}\xi|\xi\rangle-
\sum_{j=1}^{k_i}\frac{|c_j(\epsilon)|^2 \left(1-\sigma^2 c g_{\mu_{\sigma,\nu,c}}(\Phi_{\sigma,\nu,c} (\theta_i^{(j)}(\epsilon)))\right)}{ \omega_{\sigma,\nu,c}'\left(\Phi_{\sigma,\nu,c} (\theta_i^{(j)}(\epsilon))
\right)}
\right|$$ 
\begin{equation}\label{5.15}~~~~~~~~~~~~~~~~~~~~~~\leq(k_i^2+1)C\epsilon+ \sqrt{\gamma}k_i^3+\gamma .
\end{equation}
Since we have the identity  $$\langle P_{\ker(\alpha_i(N)I_N-A_NA_N^*)}\xi|\xi\rangle=\|P_{\ker( \alpha_i(N)I_N-A_NA_N^*)}\xi\|_2^2$$ and the three obvious
  convergences 
 $\lim_{\epsilon\to0}\omega_{\sigma,\nu,c}'\left(\Phi_{\sigma,\nu,c} (\theta_i^{(j)}(\epsilon))
\right)=\omega_{\sigma,\nu,c}'(\rho_{\theta_i})$, $\lim_{\epsilon\to0}g_{\mu_{\sigma,\nu,c}}\left(\Phi_{\sigma,\nu,c} (\theta_i^{(j)}(\epsilon))
\right)=g_{\mu_{\sigma,\nu,c}}(\rho_{\theta_i})$ and $\lim_{\epsilon\to0}\sum_{j=1}^{k_i}|c_j(\epsilon)|^2=1$,
relation \eqref{5.15} concludes Step B and the proof of Theorem \ref{cvev1p1}. (Note that we use (2.9) of \cite{MC2014} which is true for any $x\in \mathbb{C}\setminus \mathbb{R}$ to deduce that $1-\sigma^2 c g_{\mu_{\sigma,\nu,c}}(\Phi_{\sigma,\nu,c}(\theta_i))=
\frac{ 1}{1+ \sigma^2 cg_\nu(\theta_i)}$ by letting $x$ goes to $\Phi_{\sigma,\nu,c}(\theta_i)$).

\section*{Appendix  A}
We present  alternative versions on the one hand of the result in \cite{BaiSilver} about the lack of eigenvalues outside the support of the deterministic equivalent measure, and on the other hand of the result in \cite{MC2014} about the exact separation phenomenon. These new versions (Theorems \ref{noeigenvalue} and \ref{sep} below) deal with random variables whose imaginary and real parts are independent, but remove the technical assumptions ((1.10) and ``$b_1>0$'' in Theorem 1.1 in \cite{BaiSilver} 
and ``$\omega_{\sigma,\nu,c}(b)>0$'' in Theorem 1.2 in \cite{MC2014}). The proof of Theorem \ref{noeigenvalue} is based on the results of \cite{BC}.
The arguments of  the proof of Theorem 1.2 in \cite{MC2014} and Theorem  \ref{noeigenvalue} lead to the proof of Theorem \ref{sep}.
\begin{theorem}\label{pasde}
Consider 
\begin{equation}\label{modele}M_N=( \sigma \frac{ X_N}{\sqrt{N}}+A_N)(\sigma \frac{ X_N}{\sqrt{N}}+A_N)^*,\end{equation} and assume that 
\begin{enumerate}
 \item   $X_N =  [X_{ij}]_{1\leq i\leq n, 1\leq j\leq N}$ is a $n\times N$ random   matrix  such that 
$ [X_{ij}]_{i\geq1,j\geq 1}$ is an infinite array of  random variables which satisfy \eqref{condition} and \eqref{trois} and such that 
$ \Re(X_{ij})$, $ \Im(X_{ij})$, $(i,j)\in \mathbb{N}^2$,  are independent,  centered with variance $1/2$.
\item $A_N$ is an $n\times N$ nonrandom matrix such that $\Vert A_N \Vert$ is uniformly bounded.
\item $n\leq N$ and, as  $N$ tends to infinity, $c_N=n/N \rightarrow c\in ]0,1]$.
\item
$[x,y] $, $x<y$,  is such that 
 there exists $\delta>0$ such that for all large $N$, $ ]x-\delta; y+\delta[  \subset \mathbb{R}\setminus \rm{supp} (\mu_{\sigma,\mu _{A_N A_N^*},c_N})$
 where $\mu_{\sigma,\mu _{A_N A_N^*},c_N}$ is the  nonrandom distribution   which is characterized in terms of its Stieltjes transform which satisfies the equation \eqref{TS} where we replace $c$ by $c_N$ and  $\nu$ by $\mu _{A_N A_N^*}.$
\end{enumerate}
Then, we have
 $$\mathbb P[\mbox{for all large N}, \rm{spect}(M_N) \subset  \R \setminus [x,y] ]=1.$$
\end{theorem}

Since, in the proof of Theorem \ref{pasde}, we will use tools from free probability theory, 
for the reader's convenience, we recall the following basic definitions from free probability theory. For a thorough introduction to free probability theory, we refer to \cite{VDN}.
\begin{itemize}
\item A ${\cal C}^*$-probability space is a pair $\left({\cal A}, \tau\right)$ consisting of a unital $ {\cal C}^*$-algebra ${\cal A}$ and a state $\tau$ on ${\cal A}$ i.e a linear map $\tau: {\cal A}\rightarrow \mathbb{C}$ such that $\tau(1_{\cal A})=1$ and $\tau(aa^*)\geq 0$ for all $a \in {\cal A}$. $\tau$ is a trace if it satisfies $\tau(ab)=\tau(ba)$ for every $(a,b)\in {\cal A}^2$. A trace is said to be faithful if $\tau(aa^*)>0$ whenever $a\neq 0$. 
An element of ${\cal A}$ is called a noncommutative random variable. 
\item The noncommutative $\star$-distribution of a family $a=(a_1,\ldots,a_k)$ of noncommutative random variables in a ${\cal C}^*$-probability space $\left({\cal A}, \tau\right)$ is defined as the linear functional $\mu_a:P\mapsto \tau(P(a,a^*))$ defined on the set of polynomials in $2k$ noncommutative indeterminates, where $(a,a^*)$ denotes the $2k$-uple $(a_1,\ldots,a_k,a_1^*,\ldots,a_k^*)$.
For any selfadjoint element $a_1$ in  ${\cal A}$,  there exists a probability measure $\nu_{a_1}$ on $\mathbb{R}$ such that,   for every polynomial P, we have
$$\mu_{a_1}(P)=\int P(t) \mathrm{d}\nu_{a_1}(t).$$
Then  we identify $\mu_{a_1}$ and $\nu_{a_1}$. If $\tau$ is faithful then the  support of $\nu_{a_1}$ is the spectrum of $a_1$  and thus  $\|a_1\| = \sup\{|z|, z\in \rm{support} (\nu_{a_1})\}$. 
\item A family of elements $(a_i)_{i\in I}$ in a ${\cal C}^*$-probability space  $\left({\cal A}, \tau\right)$ is free if for all $k\in \mathbb{N}$ and all polynomials $p_1,\ldots,p_k$ in two noncommutative indeterminates, one has 
\begin{equation}\label{freeness}
\tau(p_1(a_{i_1},a_{i_1}^*)\cdots p_k (a_{i_k},a_{i_k}^*))=0
\end{equation}
whenever $i_1\neq i_2, i_2\neq i_3, \ldots, i_{k-1}\neq i_k$, $(i_1,\ldots i_k)\in I^k$, and $\tau(p_l(a_{i_l},a_{i_l}^*))=0$ for $l=1,\ldots,k$.
\item A  noncommutative random variable $x$  in a ${\cal C}^*$-probability space  $\left({\cal A}, \tau\right)$ is a standard semicircular random variable
if $x=x^*$  and for any $k\in \mathbb{N}$, $$\tau(x^k)= \int t^k d\mu_{sc}(t)$$
where $d\mu_{sc}(t)=
\frac{1}{2\pi} \sqrt{4-t^2}{\1}_{[-2;2]}(t) dt$ is the semicircular standard distribution.
\item Let $k$ be a nonnull integer number. Denote by ${\cal P}$ the set of polynomials in $2k $ noncommutative indeterminates.
A sequence of families of variables $ (a_n)_{n\geq 1}  =
(a_1(n),\ldots, a_k(n))_{n\geq 1}$ in $C^* $-probability spaces 
$\left({\cal A}_n, \tau_n\right)$ converges   in $\star$-distribution, when n goes to infinity,   to some $k$-tuple of noncommutative random variables $a=(a_1,\ldots,a_k)$ in a ${\cal C}^*$-probability space  $\left({\cal A}, \tau\right)$ if the map 
$P\in {\cal P} \mapsto
\tau_n(
P(a_n,a_n^*))$ converges pointwise towards $P\in {\cal P} \mapsto\tau(
P(a,a^*))$.
\item $k$ noncommutative random variables
$a_1(n),\ldots, a_k(n)$,  in $C^* $-probability spaces 
$\left({\cal A}_n, \tau_n\right)$, ${n\geq 1}$, are said  asymptotically free if $(a_1(n),\ldots, a_k(n))$
converges in $\star$-distribution, as n goes to infinity, to some noncommutative random variables
        $(a_1,\ldots,a_k)$ in a ${\cal C}^*$-probability space  $\left({\cal A}, \tau\right)$ where $a_1, \ldots,a_k$  are free.
\end{itemize}

We will also use the following well known  result on asymptotic freeness of random matrices.
 Let ${\cal A}_n$ be 
the algebra of $n\times n$ matrices
with complex entries    and endow this algebra with
 the normalized trace defined for any $M\in {\cal A}_n$ by 
$\tau_n(M) =\frac{1}{n}\Tr(M)$. Let us 
consider a $n\times n$ so-called standard  G.U.E matrix, i.e  a random  Hermitian matrix ${\cal G}_n =  [{\cal G}_{jk}]_{j,k=1}^n$,    where   ${\cal G}_{ii}$,
$\sqrt{2} \Re e({\cal G}_{ij})$, $\sqrt{2} \Im m({\cal G}_{ij})$, ${i<j}$ are independent centered Gaussian random variables with variance $1$.
For a fixed real number $t$ independent from n, let  $H_n^{(1)}, \ldots, H_n^{(t)}$ be deterministic $n\times n$ Hermitian matrices such that $\max_{i=1}^t\sup_n \Vert H_n^{(i)} \Vert < +\infty$ and  $(H_n^{(1)}, \ldots, H_n^{(t)})$, as a t-tuple of noncommutative random variables  in $({\cal A}_n, \tau_n)$, converges in distribution when n goes to infinity. Then, according to Theorem 5.4.5 in \cite{AGZ09},
$ \frac{{\cal G}_n}{\sqrt{n}}$ and $(H_n^{(1)}, \ldots, H_n^{(t)})$  are almost surely asymptotically free i.e
almost surely, for any polynomial P in t+1 noncommutative indeterminates,
\begin{equation}\label{sansD}\tau_n\left\{  P\left({ H_n^{(1)}},\ldots,{ H_n^{(t)}},\frac{{\cal G}_n}{\sqrt{n}}\right)\right\} \rightarrow_{n\rightarrow +\infty} \tau \left( P(h_1,\ldots,h_t,s)\right)\end{equation}
where $h_1,\ldots,h_r$ and $s$ are  noncommutative random variables  in some ${\cal C}^*$-probability space $({\cal A}, \tau)$
such that $(h_1,\ldots,h_r)$ and $s$ are free,   $s$ is a standard  semi-circular noncommutative random variable and  the distribution of $(h_1,\ldots,h_t)$
is the limiting distribution of $(H_n^{(1)}, \ldots, H_n^{(t)})$. \\

Finally,    the proof of Theorem \ref{pasde} is based on the following result which can be established by following the proof of  Theorem 1.1  in \cite{BC}. First, note that  the algebra of polynomials in non-commuting
 indeterminates $X_1,\ldots, X_k$, becomes a
$\star$-algebra by anti-linear extension of
$(X_{i_1}X_{i_2}\ldots X_{i_m})^*=X_{i_m}\ldots X_{i_2}X_{i_1}$.
\begin{theorem}\label{noeigenvalue}
 Let us consider  three independent infinite arrays of random variables,
  $ [W^{(1)}_{ij}]_{i\geq1,j\geq 1}$, $ [W^{(2)}_{ij}]_{i\geq1,j\geq 1}$ and $ [X_{ij}]_{i\geq1,j\geq 1}$ where 
 \begin{itemize}
\item for $l=1,2$,  $W^{(l)}_{ii}$,
$\sqrt{2}Re(W^{(l)}_{ij})$, $\sqrt{2} Im(W^{(l)}_{ij}), i<j$, are i.i.d  centered and bounded random variables  with variance 1 and $W^{(l)}_{ji}=\overline{W^{(l)}_{ij}}$,
\item $\{\Re (X_{ij}), \Im (X_{ij}), i\in \mathbb{N}, j  \in \mathbb{N}\}$ are independent   centered random variables   with variance $1/2$ and satisfy \eqref{condition} and \eqref{trois}.
\end{itemize}  
  For any $(N,n)\in \mathbb{N}^2$, define the $(n+N)\times (n+N) $ matrix: \begin{equation}\label{Wigner}W_{n+N}=\begin{pmatrix}  W_n^{(1)} & X_N \\ X_N^* &  W_N^{(2)} \end{pmatrix}\end{equation}
  where $X_N=[X_{ij}]_{\tiny \begin{array}{ll}1\leq i\leq n\\1\leq j\leq N\end{array}}, \; W^{(1)}_n= [W^{(1)}_{ij}]_{1\leq i,j\leq n},\;
  W^{(2)}_N= [W^{(2)}_{ij}]_{1\leq i,j\leq N}$. \\
  Assume that $n=n(N)$ and $\lim_{N\rightarrow +\infty}\frac{n}{N}=c \in ]0,1].$\\
  Let $t$ be a fixed integer number and  $P$ be a selfadjoint  polynomial  in $t+1$ noncommutative  indeterminates.\\ For any $N \in \mathbb{N}^2$,  let  $(B_{n+N}^{(1)},\ldots,B_{n+N}^{(t)})$  be    a $t-$tuple of  $(n+N)\times (n+N)$   deterministic Hermitian matrices    such that
  for any $u=1,\ldots,t$, $ \sup_{N} \Vert B_{n+N}^{(u)} \Vert< \infty$.
Let $({\cal A},  \tau)$ be a $C^*$-probability space
 equipped with a faithful tracial state and 
 $s$ be  a standard semi-circular noncommutative random variable in  $({\cal A},  \tau)$.
Let $b_{n+N}=(b_{n+N}^{(1)},\ldots,b_{n+N}^{(t)})$ be a t-tuple of noncommutative  selfadjoint  random variables which is free from  $s$ in $({\cal A},\tau)$ and such that the distribution of  $b_{n+N}$  in $({\cal A},\tau)$ coincides with the distribution of $(B_{n+N}^{(1)},\ldots, B_{n+N}^{(t)})$ in $({ M}_{n+N}(\mathbb{C}), \frac{1}{n+N}\Tr)$. \\ Let $[x,y]$ be a real interval  such  that there exists $\delta>0$ such that, for any large  $N$,   $[x-\delta,y+\delta]$ lies outside the support of the distribution of  the noncommutative random variable  $ P\left(s, b_{n+N}^{(1)},\ldots,b_{n+N}^{(t)}\right)$ in $({\cal A},\tau)$.
  Then, almost surely, for all large N, 
  $$ \rm{spect}P\left(\frac{{ W}_{n+N}}{\sqrt{n+N}}, B_{n+N}^{(1)},\ldots,B_{n+N}^{(t)})\right) \subset  \R \setminus [x,y].$$

\end{theorem}
\begin{proof}
We start by  checking that a truncation and Gaussian convolution procedure as  in Section 2 of \cite{BC} can be handled for such a matrix  as defined by \eqref{Wigner},
  to reduce the problem to 
 a fit framework 
 where,
\begin{itemize}
\item[(H)]  for any $N$,  
 $(W_{n+N})_{ii}$, $\sqrt{2}Re((W_{n+N})_{ij})$, $\sqrt{2} Im((W_{n+N})_{ij}),  i<j, i \leq n+N,\; j\leq n+N$,  are independent, centered random variables  with variance 1, which  satisfy a Poincar\'e inequality with common fixed constant  $C_{PI}$.
 \end{itemize}
  \noindent Note that, according to Corollary 3.2 in \cite{Ledoux01},  (H) implies that  for any $p\in \mathbb{N}$,  \begin{equation}\label{moments}  \sup_{N\geq 1} \sup_{1\leq i,j\leq n+N} \mathbb{E}\left(\vert  (W_{n+N})_{ij}\vert^p\right) <+\infty.\end{equation} 
  \begin{remark}\label{2.1} Following the proof of Lemma 2.1 in \cite{BC}, one can establish that, if $(V_{ij})_{i\geq 1, j\geq 1}$ is an infinite array of random variables such that $\{\Re (V_{ij}), \Im (V_{ij}), i\in \mathbb{N}, j  \in \mathbb{N}\}$ are independent   centered random variables  which satisfy \eqref{condition} and \eqref{trois}, then
 almost surely we have $$\limsup_{N\rightarrow +\infty} \left\| \frac{Z_{n+N}}{\sqrt{N+n}}\right\| \leq 2\sigma^*$$ where 
 $$Z_{n+N}=\begin{pmatrix}  (0) & V_N \\ V_N^* &  (0) \end{pmatrix} \; \mbox{ with}\;  V_N=[V_{ij}]_{\tiny \begin{array}{ll}1\leq i\leq n\\1\leq j\leq N\end{array}} \mbox{and}\; \sigma^*=\left\{\sup_{(i,j)\in \mathbb{N}^2}\mathbb{E}(\vert V_{ij}\vert^2)\right\}^{1/2}.$$
 \end{remark}
 Then, following the rest of the  proof of Section 2 in \cite{BC}, one can prove that for any  polynomial $P$ in $1+t$ noncommutative variables, there exists some constant $L>0$ such that the following holds.  Set $\theta^*=\sup_{i,j}\mathbb{E}\left(\left|X_{ij}\right|^3\right)$. For any $0<\epsilon<1$, there exist $C_\epsilon>8\theta^* $ (such that  $C_\epsilon >\max_{l=1,2} \vert W^{(l)}_{11}\vert $ a.s.) and $\delta_\epsilon>0$ such that almost surely for all large $N$, \begin{equation}\label{fit} \left\| P\left(\frac{W_{n+N}}{\sqrt{n+N}},B_{n+N}^{(1)},\ldots,B_{n+N}^{(t)}\right)- P\left(\frac{\tilde W_{n+N}^{C_\epsilon,\delta_\epsilon}}{\sqrt{n+N}}, B_{n+N}^{(1)},\ldots,B_{n+N}^{(t)}\right)\right\|\leq L \epsilon,\end{equation}
 where, for any  $C>8 \theta^*$ such that $C>\max_{l=1,2} \vert W^{(l)}_{11}\vert $ a.s., 
 and for any $\delta>0$, $\tilde W_{N+n}^{C,\delta}$ is a $(n+N)\times (n+N) $ matrix  which is defined as follows.
 Let $({\cal G}_{ij})_{i\geq 1, j\geq 1}$ be an infinite array which is  independent of  $\{X_{ij}, W^{(1)}_{ij}, W^{(2)}_{ij}, (i,j)\in \mathbb{N}^2\}$ and such that $\sqrt{2} \Re e{\cal G}_{ij}$, $ \sqrt{2} \Im m{\cal G}_{ij}$, $i<j$, ${\cal G}_{ii}$,  are independent centred standard real gaussian variables and ${\cal G}_{ij}=\overline{\cal G}_{ji}$.
Set   ${\cal G}_{n+N}= [{\cal G}_{ij}]_{1\leq i,j \leq n+N }$ and define  $X_N^C=[X_{ij}^C]_{\tiny \begin{array}{ll}1\leq i\leq n\\1\leq j\leq N\end{array}}$ as in \eqref{defxc}. 
 Set  $$\tilde W_{n+N}^C=\begin{pmatrix}  W_n^{(1)} & X_N^C \\ (X_N^C)^* &  W_N^{(2)} \end{pmatrix}\; \mbox{and} \; \tilde W_{N+n}^{C,\delta}= \frac{ \tilde W_{n+N}^C +\delta {\cal G}_{n+N}}{\sqrt{1+\delta^2}}.$$
  $\tilde W_{N+n}^{C,\delta}$ satisfies (H) (see the end of Section 2 in \cite{BC}).
\eqref{fit} readily yields that it is sufficient to prove Theorem \ref{noeigenvalue} for  $\tilde W_{N+n}^{C,\delta}$.\\
Therefore, assume now that $W_{N+n}$ satisfies (H).
As explained in Section 6.2 in \cite{BC}, to establish Theorem \ref{noeigenvalue}, it is sufficient to prove that 
 for all $m \in \mathbb{N}$, all self-adjoint matrices $\gamma, \alpha, \beta_1, \ldots, \beta_t$ of size $m\times m$ and
all $\epsilon >0$, almost surely, for all large $N$, we have\\

\noindent $
spect(\gamma \otimes I_{n+N} +  \alpha\otimes \frac{W_{n+N}}{\sqrt{n+N}}+  \sum_{u=1}^t \beta_u \otimes B_{n+N}^{(u)})$ \begin{equation} \label{spectre3} \subset
spect(\gamma \otimes 1_{\cal A} +  \alpha \otimes s+ \sum_{u=1}^t \beta_u \otimes b_{n+N}^{(u)}) + ]-\epsilon, \epsilon[.
\end{equation}
(\eqref{spectre3} is the analog of  Lemma 1.3 for $r=1$ in \cite{BC}).
 Finally, one can prove \eqref{spectre3} by  following Section 5 in \cite{BC}.
\end{proof}
We will need the following lemma in the proof of Theorem \ref{pasde}.
\begin{lemme}\label{interpretation} Let $A_N$ and $c_N$ be  defined as in Theorem \ref{pasde}. Define the following $(n+N)\times (n+N) $ matrices: $P=\begin{pmatrix} I_n & (0) \\ (0) & (0) \end{pmatrix}$ and $Q=\begin{pmatrix} (0) & (0) \\ (0) & I_N\end{pmatrix}$
and ${\bf A}=\begin{pmatrix}  (0) & A_N \\ (0) & (0) \end{pmatrix}$.
 Let  $s,p_N,q_N, {\bf a}_N$ be noncommutative random variables in some $\mathcal{C}^*$-probability space $\left(  {\cal A},  \tau\right)$ such that $s$ is a standard  semi-circular  variable which is free with $(p_N,q_N, {\bf a}_N)$ and  the $\star$-distribution of $({\bf  A},P,Q)$ in $\left(M_{N+n}(\mathbb{C}),\frac{1}{N+n} \Tr\right)$ coincides with the $\star$-distribution of 
$({\bf a}_N, p_N,q_N)$ in $\left(  {\cal A},  \tau\right). $
Then, for any $\epsilon \geq 0$, the distribution of $ ({\sqrt{1+c_N}}\sigma p_N s q_N+ {\sqrt{1+c_N}}\sigma q_N  s p_N + {\bf a}_N+ {\bf a}_N^*)^2 +\epsilon p_N$
is $\frac{n}{N+n} T_\epsilon \star \mu_{\sigma, \mu_{A_NA_N^*}, c_N} +\frac{n}{N+n}   \mu_{\sigma, \mu_{A_NA_N^*}, c_N}+\frac{N-n}{N+n} \delta_{0}$
where ${T_\epsilon} {\star}  \mu_{\sigma, \mu_{A_NA_N^*}, c_N}$ is the pushforward of $ \mu_{\sigma, \mu_{A_NA_N^*}, c_N}$ by the map $z\mapsto z+\epsilon$. 

\end{lemme}
\begin{proof}
{\it Here $N$ and $n$ are fixed}. Let $k\geq 1$ and $C_k$ be the $k\times k$ matrix defined by
$$ C_k= \begin{pmatrix} \hspace*{-0.4cm} (0)&{1} \\ \hspace*{-0.1cm} \; \;  \; \; \; \;  \;\mbox{ \reflectbox{$\ddots$}}  \\ \hspace*{-0.4cm}{1} & (0)\end{pmatrix}.$$ 
 Define the $k(n+N)\times k(n+N)$ matrices 
$$\hat A_k= C_k\otimes {\bf A},\; \hat P_k=I_k\otimes P, \; \hat Q_k= I_k\otimes Q.$$
For any $k \geq 1$,  the $\star$-distributions of $(\hat A_k, \hat P_k, \hat Q_k)$ in $( M_{k(N+n)}(\mathbb{C}), \frac{1}{k(N+n)}\Tr)$ 
and  $({\bf A},  P,  Q)$ in $( M_{(N+n)}(\mathbb{C}), \frac{1}{(N+n)}\Tr)$ respectively, coincide.
Indeed, let ${\cal K}$ be a noncommutative monomial in $\mathbb{C}\langle X_1,X_2,X_3,X_4\rangle$ and denote by $q$ the total number of occurrences of $X_3$ and $X_4$ in ${\cal K}$. We have 
$${\cal K}(\hat P_k, \hat Q_k, \hat A_k, \hat A_k^*)=C_k^q \otimes {\cal K}(P,Q,{\bf A}, {\bf A}^*),$$
so that $$\frac{1}{k(n+N)} \Tr  \left[{\cal K}(\hat P_k, \hat Q_k, \hat A_k, \hat A_k^*)\right]= \frac{1}{k}\Tr (C_k^q) \frac{1}{(n+N)}\Tr \left[{\cal K}(P,Q,{\bf A}, {\bf A}^*)\right].$$
Note that if $q$ is even then $C_k^q=I_k$ so that 
\begin{equation}\label{egalite}\frac{1}{k(n+N)} \Tr \left[{\cal K}(\hat P_k, \hat Q_k, \hat A_k, \hat A_k^*)\right]=\frac{1}{(n+N)}\Tr\left[ {\cal K}(P,Q,{\bf A}, {\bf A}^*)\right].\end{equation}
Now, assume that $q$ is odd. Note that $PQ=QP=0, \;{\bf A}Q={\bf A}, \; Q{\bf A}=0, \; {\bf A}P=0$ and $P{\bf A}={\bf A}$
(and then $Q{\bf A}^*={\bf A}^*, \; {\bf A}^*Q=0,\, P{\bf A}^*=0$ and ${\bf A}^* P={\bf A}^*$). Therefore, if at least one of the terms $X_1 X_2$, $X_2 X_1$,
$X_2 X_3$, $ X_3 X_1$, $X_4 X_2$ or $X_1 X_4$ appears  in the noncommutative product in ${\cal K}$,  then 
$ {\cal K}(P,Q,{\bf A}, {\bf A}^*)=0,$ so that \eqref{egalite} still holds. Now, if none of the terms $X_1 X_2$, $X_2 X_1$,
$X_2 X_3$, $ X_3 X_1$, $X_4 X_2$ or $X_1 X_4$ appears  in  the noncommutative product in ${\cal K}$,  then we have 
${\cal K}(P,Q,{\bf A}, {\bf A}^*)=\tilde {\cal K}({\bf A}, {\bf A}^*)$ for some noncommutative monomial  $\tilde {\cal K}\in \mathbb{C} \langle X,Y\rangle $ with 
degree $q$. 
Either  the noncommutative product in $\tilde {\cal K}$ contains a term such as $X^p$ or $Y^p$ for some $p \geq 2$ and then, since ${\bf A}^2= ({\bf A}^*)^2=0$,
we have $\tilde {\cal K}({\bf A}, {\bf A}^*)=0$,
or $\tilde {\cal K}(X,Y)$ is one of the monomials $ (XY)^{\frac{q-1}{2}}X$ or  $Y(XY)^{\frac{q-1}{2}}$. In both cases, we have $\Tr \tilde {\cal K}({\bf A}, {\bf A}^*)=0$ and \eqref{egalite} still holds.\\
Now, define  the $k(N+n)\times k(N+n)$ matrices
$$\tilde P_k=\begin{pmatrix} I_{kn} & (0) \\ (0) & (0) \end{pmatrix}, \; \;\tilde Q_k= \begin{pmatrix} (0) & (0) \\ (0) & I_{kN}\end{pmatrix},\;
\tilde A_k=\begin{pmatrix} (0) & \check{A} \\ (0) & (0) \end{pmatrix} $$
where $\check{ A}$  is the $kn\times kN$ matrix defined by $$\check{ A}=\begin{pmatrix} (0)& A_N \\ \; \;  \; \; \; \;  \;\mbox{ \reflectbox{$\ddots$}}  \\  A_N & (0)\end{pmatrix}.$$
It is clear that there exists a real orthogonal $k(N+n)\times k(N+n)$ matrix $O$ such that $\tilde P_k=O\hat P_k O^*$, $\tilde Q_k=O\hat Q_k O^*$ and 
$\tilde A_k=O\hat A_k O^*$. This readily yields that 
 the noncommutative  $\star$-distributions of $(\hat A_k, \hat P_k, \hat Q_k)$ and $({\tilde  A}_k,  \tilde P_k,  \tilde Q_k)$ in $( M_{k(N+n)}(\mathbb{C}), \frac{1}{k(N+n)}\Tr)$  coincide. Hence, for any $k\geq 1$, the distribution of  $({\tilde  A}_k,  \tilde P_k,  \tilde Q_k)$ in $( M_{k(N+n)}(\mathbb{C}), \frac{1}{k(N+n)}\Tr)$  coincides with the distribution of $({\bf a}_N,p_N,q_N)$ in $\left(  {\cal A},  \tau\right). $
By Theorem 5.4.5 in \cite{AGZ09}, it readily follows that the distribution of $({\sqrt{1+c_N}}\sigma p_N s q_N+ {\sqrt{1+c_N}}\sigma q_N  s p_N + {\bf a}_N+ {\bf a}_N^*)^2 +\epsilon p_N$ is the almost sure limiting distribution, when $k$ goes to infinity, of 
$({\sqrt{1+c_N}}\sigma \tilde P_k\frac{{\cal G} }{\sqrt{k(N+n)}}\tilde  Q_k+ {\sqrt{1+c_N}} \sigma \tilde Q_k \frac{{\cal G} }{\sqrt{k(N+n)}}\tilde P_k+\tilde A_k+\tilde A_k^*)^2+\epsilon \tilde P_k$ in  $( M_{k(N+n)}(\mathbb{C}), \frac{1}{k(N+n)}\Tr)$,
where ${\cal G}$ is a $k(N+n)\times k(N+n)$ GUE matrix with entries with variance 1. 
Now, note that 
$$\left[{\sqrt{1+c_N}}\sigma\left\{\tilde P_k \frac{{\cal G} }{\sqrt{k(N+n)}} \tilde Q_k
+\tilde Q_k \frac{{\cal G} }{\sqrt{k(N+n)}}\tilde P_k\right\}  +\tilde A_k +\tilde A_k^*\right]^2  +\epsilon \tilde P_k  $$
 $$=
\begin{pmatrix} (\sigma \frac{{\cal G}_{kn\times kN}}{\sqrt{kN}}+\check{ A})(\sigma \frac{{\cal G}_{kn\times kN}}{\sqrt{kN}}+\check{ A})^*+\epsilon I_{kn} &(0)\\ (0)& (\sigma \frac{{\cal G}_{kn\times kN}}{\sqrt{kN}}+\check{A})^*(\sigma \frac{{\cal G}_{kn\times kN}}{\sqrt{kN}}+\check{A})
\end{pmatrix}$$
where ${\cal G}_{kn\times kN}$ is the upper right $kn\times kN$ corner of $ {\cal G}$.
Thus, noticing that $\mu_{\check{ A}\check{ A}*}=\mu_{A_NA_N^*}$, the lemma follows from \cite{DozierSilver}.
\end{proof}
\noindent {\it Proof of Theorem \ref{pasde}}.
Let $W$ be 
a $(n+N)\times (n+N) $ matrix as defined by \eqref{Wigner} in Theorem \ref{noeigenvalue}.
Note that, with the notations of Lemma \ref{interpretation}, for any $\epsilon\geq 0$, \\

$
\begin{pmatrix} (\sigma \frac{X_N}{\sqrt{N}}+A_N)(\sigma \frac{X_N}{\sqrt{N}}+A_N)^*+\epsilon I_n &(0)\\ (0)& (\sigma \frac{X_N}{\sqrt{N}}+A_N)^*(\sigma \frac{X_N}{\sqrt{N}}+A_N)
\end{pmatrix}$
\begin{eqnarray*}&=& \begin{pmatrix} (0)&(\sigma \frac{X_N}{\sqrt{N}}+A_N) \\(\sigma \frac{X_N}{\sqrt{N}}+A_N)^* &(0) 
\end{pmatrix}^2 +\epsilon P
\\
&=& \left({\sqrt{1+c_N}}P\frac{\sigma W}{\sqrt{N+n}}Q+ {\sqrt{1+c_N}}Q\frac{\sigma W}{\sqrt{N+n}}P+{\bf A}+{\bf  A}^*\right)^2+\epsilon P.
\end{eqnarray*}
Thus, for any $\epsilon \geq 0$,\\

$ \rm{spect}\left\{(\sigma \frac{X_N}{\sqrt{N}}+A)(\sigma \frac{X_N}{\sqrt{N}}+A)^*+\epsilon I_n\right\}$
\begin{equation}\label{inclusion}\subset  \rm{spect}\left\{\left({\sqrt{1+c_N}}P\frac{\sigma W}{\sqrt{N+n}}Q+ {\sqrt{1+c_N}}Q\frac{\sigma W}{\sqrt{N+n}}P+{\bf A}+{\bf A}^*\right)^2+\epsilon P\right\}.\end{equation}
 Let $[x,y] $ be such that 
 there exists $\delta>0$ such that for all large $N$, $ ]x-\delta; y+\delta[  \subset \mathbb{R}\setminus \rm{supp} (\mu_{\sigma,\mu _{A_N A_N^*},c_N})$.
  \begin{itemize}
 \item[(i)]
  Assume $x>0.$ Then, according to Lemma \ref{interpretation} with $\epsilon=0$,  there exists $\delta'>0$ such that for all large  $n$,  $ ]x-\delta'; y+\delta'[$ is outside the support of the distribution of 
  $ ({\sqrt{1+c_N}}\sigma p_N s q_N+ {\sqrt{1+c_N}}\sigma q_N  s p_N + {\bf a}_N+ {\bf a}_N^*)^2 $.
We readily deduce that almost surely for all large N,  according to  Theorem \ref{noeigenvalue}, there is no eigenvalue of $({\sqrt{1+c_N}}P\frac{\sigma W}{\sqrt{N+n}}Q+ {\sqrt{1+c_N}}Q\frac{\sigma W}{\sqrt{N+n}}P+{\bf A}+{\bf A}^*)^2 $ in $[x,y]$. Hence, by \eqref{inclusion} with $\epsilon=0$, almost surely for all large N,  there is no eigenvalue of $M_N$ in $[x,y].$
\item[(ii)] Assume $x= 0$ and $y > 0$. There exists $0<\delta'<y$ such that $[0,3\delta']$ is for all large $N$ outside  the support 
of  $\mu_{\sigma, \mu_{A_NA_N^*}, c_N}$. Hence, according to Lemma \ref{interpretation}, $[\delta'/2,3\delta']$ is outside the support of the distribution of $ ({\sqrt{1+c_N}}\sigma p_N s q_N+ {\sqrt{1+c_N}}\sigma q_N  s p_N + {\bf a}_N+ {\bf a}_N^*)^2 +\delta' p_N$. Then,  almost surely for all large N,  according to  Theorem \ref{noeigenvalue}, there is no eigenvalue of $({\sqrt{1+c_N}}P\frac{\sigma W}{\sqrt{N+n}}Q+ {\sqrt{1+c_N}}Q\frac{\sigma W}{\sqrt{N+n}}P+{\bf A}+{\bf A}^*)^2 +\delta' P $ in $[\delta',2\delta']$ and thus, by \eqref{inclusion},  no eigenvalue of $ (\sigma \frac{X}{\sqrt{N}}+A_N)(\sigma \frac{X_N}{\sqrt{N}}+A_N)^*+\delta' I_n$ in $[\delta',2\delta']$. It readily follows that, almost surely for all large N, there is no eigenvalue of 
$ (\sigma \frac{X_N}{\sqrt{N}}+A_N)(\sigma \frac{X_N}{\sqrt{N}}+A_N)^*$ in $[0,\delta']$. Since moreover, according to (i),  almost surely for all large N, there is no eigenvalue of 
$ (\sigma \frac{X_N}{\sqrt{N}}+A_N)(\sigma \frac{X_N}{\sqrt{N}}+A_N)^*$ in $[\delta',y]$, we can conclude that there is no eigenvalue of $M_N$ in $[x,y]$.
 \end{itemize}
The proof of Theorem \ref{pasde} is now complete. ~~~~~~~~~~~~~~~~~~~~~~~~~~~~~~~~~~~~~~~~~~~~~~~~~~~~~~~~~~~~~~~~~~~~~~~~~$\Box$\\

We are now in a position to establish the following exact separation phenomenon.
\begin{theorem}\label{sep} Let $M_n$ as in \eqref{modele} with   assumptions [1-4] of Theorem \ref{pasde}.
Assume moreover that the empirical spectral measure $\mu_{A_NA_N^*}$ of $A_NA_N^*$ converges weakly to some probability measure $\nu$.
 Then for  $N$ large enough, \begin{equation}\label{lemme31}\omega_{{\sigma,\nu,c}}([x,y])=[\omega_{{\sigma,\nu,c}}(x);\omega_{{\sigma,\nu,c}}(y)] \subset \mathbb{R} \setminus \mbox{supp}(\mu _{A_N A_N^*}),\end{equation}
 where $\omega_{\sigma,\nu,c}$ is defined in \eqref{defom}.
With the convention that $\lambda _0(M_N)=\lambda _0(A_NA_N^*)=+\infty $ and 
$\lambda _{n+1}(M_N)=\lambda _{n+1}(A_NA_N^*)=-\infty $, for $N$ large enough, let  $i_N\in \{0,\ldots,n\}$
be such that
\begin{equation}\label{iN2}\lambda_{i_N+1}(A_N A_N^*) <\omega_{{\sigma,\nu,c}}(x) \mbox{~~ and ~~} \lambda_{i_N}(A_N A_N^*) > \omega_{{\sigma,\nu ,c}}(y).\end{equation}
Then
\begin{equation}\label{sepeq}P[\mbox{for all large N}, \lambda_{i_N+1}(M_N) <x\mbox{~and} ~ \lambda_{i_N}(M_N)> 
y] = 1.\end{equation}
\end{theorem}
\begin{remark} \label{incl}  Since $\mu_{\sigma,\mu _{A_N A_N^*},c_N}$ converges weakly towards $\mu_{\sigma,\nu,c}$  assumption 4. implies that 
$\forall 0< \tau< \delta$, $[x-\tau; y+\tau] \subset \mathbb{R} \setminus  \rm{supp}~ \mu_{\sigma,\nu,c}$.
\end{remark}
\begin{proof}
\eqref{lemme31} is proved in Lemma 3.1 in \cite{MC2014}.
\begin{itemize} \item  If   $\omega_{{\sigma,\nu,c}}(x)< 0$, then  $i_N=n $ in \eqref{iN2} and moreover we have, for all large N, $\omega_{{\sigma,\mu _{A_N A_N^*},c_N}}(x)<0$. According to Lemma 2.7 in \cite{MC2014}, we can deduce  that, for all large $N$,  $[x,y] $ is on the left hand side of the support of $\mu_{\sigma,\mu _{A_N A_N^*},c_N}$ so that $]-\infty; y+\delta]$ is on the left hand side of the support of $\mu_{\sigma,\mu _{A_N A_N^*},c_N}$. 
Since $[-\vert y \vert -1,y]$ satifies the assumptions of Theorem \ref{pasde}, 
we readily deduce that almost surely, for all large $N$, $\lambda_{n}(M_N)> y.$  Hence \eqref{sepeq} holds true.
\item  If   $\omega_{{\sigma,\nu,c}}(x)\geq  0$, 
we first explain why it is sufficient to prove  \eqref{sepeq} for $x$ such that $\omega_{{\sigma,\nu,c}}(x)>0.$ 
Indeed, assume for a while that \eqref{sepeq}
is true whenever  $\omega_{{\sigma,\nu,c}}(x)>0$.
Let us consider any interval $[x,y]$ satisfying condition 4. of Theorem  \ref{pasde} and such that $\omega_{{\sigma,\nu,c}}(x)= 0$;  then $i_N=n $ in \eqref{iN2}. According to Proposition \ref{caractfinale},   $\omega_{{\sigma,\nu,c}}(\frac{x+y}{2})> 0$ and then  almost surely for all large N,
 $\lambda_n (M_N)>y.$ Finally, sticking to the proof of Theorem 1.2 in \cite{MC2014} leads to \eqref{sepeq} for $x$ such that $\omega_{{\sigma,\nu,c}}(x)>0.$ 
\end{itemize}

\end{proof}

\section*{Appendix  B}

We first  recall some basic properties of the resolvent (see \cite{KKP96}, \cite{CD07}).
\begin{lemme} \label{lem0}
For a $N \times N$ Hermitian matrix $M$, 
for any $z \in \C\setminus {\rm spect}(M)$, 
we denote by $G(z) := (zI_N-M)^{-1}$ the resolvent of $M$.\\
Let $z \in \C\setminus \R$, 
\begin{itemize}
\item[(i)] $\Vert G(z) \Vert \leq |\Im z|^{-1}$. 
\item[(ii)] $\vert G(z)_{ij} \vert \leq |\Im z|^{-1}$ for all $i,j = 1, \ldots , N$.
\item[(iii)] $G(z)M=MG(z) =-I_N +zG(z)$. 
\end{itemize}
Moreover, for any $N \times N$ Hermitian matrices $M_1$ and $ M_2$, 
$$(zI_N-M_1)^{-1}-(zI_N-M_2)^{-1}=(zI_N-M_1)^{-1}(M_1-M_2)(zI_N-M_2)^{-1}.$$
\end{lemme}

\noindent The following technical lemmas are fundamental in the approach of the present paper. 
\begin{lemme}\label{approxpoisson}[Lemma 4.4 in \cite{BBCF15}]
Let $h: \mathbb{R}\rightarrow \mathbb{R}$ be a continuous function  with compact support. Let $B_N$ be a  $N\times N $ Hermitian matrix and $C_N$  be a  $N\times N $ matrix .
Then \begin{equation}\label{sansE}\Tr \left[h(B_N) C_N\right]= - \lim_{y\rightarrow 0^{+}}\frac{1}{\pi} \int \Im \Tr \left[(t+iy-B_N)^{-1}C_N\right] h(t) dt. \end{equation}
Moreover, if $B_N$ is random, we also have
\begin{equation}\label{avecE}\mathbb{E}\Tr \left[h(B_N) C_N\right]= - \lim_{y\rightarrow 0^{+}}\frac{1}{\pi} \int \Im \mathbb{E}\Tr \left[(t+iy-B_N)^{-1}C_N\right] h(t) dt. \end{equation}
\end{lemme}
\begin{lemme}\label{HT} Let $f$ be an analytic function on $\C\setminus \R$ such that there exist some polynomial $P$ with nonnegative coefficients, and some positive real number $\alpha$ such that
\begin{equation*}\label{nestimgdif}
\forall z \in \mathbb{C}\setminus \mathbb{R},~~\vert f(z)\vert \leq (\vert z\vert +1)^\alpha P(\vert \Im z\vert ^{-1}).
\end{equation*} Then, for any  $h$  in $\cal C^\infty (\R, \R)$ with compact support, there exists some constant $\tau$ depending only on $h$,  $ \alpha$ and $P$ such that 
$$\limsup _{y\rightarrow 0^+}\vert \int _\R h (x)f(x+iy)dx\vert < \tau.$$
\end{lemme}
\noindent We refer the reader to the Appendix of \cite{CD07} 
where it is proved using the ideas of \cite{HaaThor05}.\\

Finally, we recall some  facts on Poincar\'e inequality.
A probability measure $\mu$ on $\mathbb{R}$ is said to satisfy the Poincar\' e inequality with constant $C_{PI}$ if
 for any
${\cal C}^1$ function $f: \R\rightarrow \C$  such that $f$ and
$f' $ are in $L^2(\mu)$,
$$\mathbf{V}(f)\leq C_{PI}\int  \vert f' \vert^2 d\mu ,$$
\noindent with $\mathbf{V}(f) = \int \vert
f-\int f d\mu \vert^2 d\mu$. \\
 We refer the reader to \cite{BobGot99} for a characterization
 of  the measures on  $\mathbb{R}$ which satisfy a Poincar\'e inequality. 

\noindent  If the law of a random variable $X$ satisfies the Poincar\'e inequality with constant $C_{PI}$ then, for any fixed $\alpha \neq 0$, the law of $\alpha X$ satisfies the Poincar\'e inequality with constant $\alpha^2 C_{PI}$.\\
Assume that  probability measures $\mu_1,\ldots,\mu_M$ on $\mathbb{R}$ satisfy the Poincar\'e inequality with constant $C_{PI}(1),\ldots,C_{PI}(M)$ respectively. Then the product measure $\mu_1\otimes \cdots \otimes \mu_M$ on $\mathbb{R}^M$ satisfies the Poincar\'e inequality with constant $\displaystyle{C_{PI}^*=\max_{i\in\{1,\ldots,M\}}C_{PI}(i)}$ in the sense that for any differentiable function $f$ such that $f$ and its gradient ${\rm grad} f$ are in $L^2(\mu_1\otimes \cdots \otimes \mu_M)$,
$$\mathbf{V}(f)\leq C_{PI}^* \int \Vert {\rm grad} f \Vert_2 ^2 d\mu_1\otimes \cdots \otimes \mu_M$$
\noindent with $\mathbf{V}(f) = \int \vert
f-\int f d\mu_1\otimes \cdots \otimes \mu_M \vert^2 d\mu_1\otimes \cdots \otimes \mu_M$ (see Theorem 2.5 in  \cite{GuZe03}) .

\begin{lemme}\label{zitt}[Theorem 1.2 in \cite{BGMZ}]
Assume that the distribution of  a random variable $X$ is  supported in  $[-C;C]$ for some constant $C>0$. Let $g$ be an independent standard real  Gaussian random variable. Then  $X+\delta g$ satisfies a Poincar\'e inequality with constant 
$C_{PI}\leq \delta^2 \exp \left( 4C^2/\delta^2\right)$.
\end{lemme}

\noindent {\bf Acknowlegements.}
The author is very grateful to  Charles Bordenave and Serban Belinschi for several fruitful  discussions and thanks Serban Belinschi    for pointing out  Lemma \ref{gnmoinsgmu}. The author also wants to thank an anonymous referee
who provided a much simpler proof of Lemma \ref{majpre} and encouraged the author to establish  the results for non diagonal perturbations, which led to an overall improvement of the paper.


\begin{thebibliography}{10}

\bibitem{AGZ09}
G.~Anderson, A.~Guionnet, and O.~Zeitouni.
\newblock {\em An Introduction to Random Matrices}.
\newblock Cambridge University Press, 2009.










\bibitem{BaiSil06}
Z.~D. Bai and J.~W. Silverstein.
\newblock Spectral Analysis of of large-dimensional random matrices.
\newblock Mathematics Monograph Series 2, Science Press Beijing 2006.

\bibitem{BaiSilver}
Z.~Bai and J.~W. Silverstein.
\newblock No eigenvalues outside the support of the limiting spectral
  distribution of information-plus-noise type matrices.
\newblock {\em Random Matrices Theory Appl.}, 1(1):1150004, 44, 2012.











\bibitem{BGMZ}
J.B. Bardet, N. Gozlan, F. Malrieu and  P.-A. Zitt. Functional inequalities for Gaussian convolutions of compactly supported measures: explicit bounds and dimension dependence. {\em ArXiv e-prints}: 1507.02389.

\bibitem{BBCF15}
S.~T. Belinschi,  H.~Bercovici, M. Capitaine and M. F\'evrier.
Outliers in the spectrum of large deformed unitarily invariant models 
\newblock {\em ArXiv e-prints}: 1412.4916, to appear in {\em Ann. Probab.} 
\bibitem{BC} S.~T. Belinschi,  M. Capitaine. Spectral properties of polynomials in independent Wigner  and  deterministic matrices
\newblock {\em ArXiv e-prints}: 1611.07440, to appear in {\it J. Funct. Anal.}.
\bibitem{BGRao09}
F.~{Benaych-Georges} and R.~N. {Rao}.
\newblock {The eigenvalues and eigenvectors of finite, low rank perturbations
  of large random matrices}.\newblock {\em Adv. in Math.}, 227(1):494--521, 2011.

\bibitem{BGRao10}
F.~{Benaych-Georges} and R.~N. {Rao}.
\newblock {The singular values and vectors of low rank perturbations
  of large rectangular random matrices}.
\newblock {\em ArXiv e-prints}: 1103.2221, 2011.


\bibitem{BobGot99}
S.~G. Bobkov and F.~G{\"o}tze.
\newblock Exponential integrability and transportation cost related to
  logarithmic {S}obolev inequalities.
\newblock {\em J. Funct. Anal.}, 163(1):1--28, 1999.








\bibitem{MCJTP}
M.~Capitaine.
Additive/multiplicative free subordination
property and limiting eigenvectors of spiked
additive deformations of Wigner matrices and
spiked sample covariance matrices
{\em Journal of Theoretical Probability}, Volume 26 (3) (2013), 595--648.


\bibitem{MC2014}
M.~Capitaine.
\newblock Exact separation phenomenon for the eigenvalues of large Information-Plus-Noise type matrices. Application to spiked models.
\newblock {\em Indiana Univ. Math. J.}, 63 (6): 1875--1910, 2014.

\bibitem{CD07}
M.~Capitaine and C.~Donati-Martin.
\newblock Strong asymptotic freeness for {W}igner and {W}ishart matrices.
\newblock {\em Indiana Univ. Math. J.}, 56(2):767--803, 2007.
 \bibitem{CDM} M.~Capitaine and C.~Donati-Martin.
Spectrum of deformed random matrices 
 and free probability.  To appear in SMF volume Panoramas et Synth\`eses (2016). 
 
 
 \bibitem{CSBD}
R.~Couillet, J.~W. Silverstein, Z.~Bai, and M.~Debbah.
\newblock Eigen-inference for energy estimation of multiple sources.
\newblock {\em IEEE Trans. Inform. Theory}, 57(4):2420--2439, 2011.








\bibitem{DozierSilver} 
R.B. Dozier and J.W. Silverstein.
\newblock On the empirical  distribution  of eigenvalues  of large dimensional information-plus-noise type  matrices.
\newblock {\em J. Multivariate. Anal.}, vol. 98, no. 4: 678--694 , 2007.

\bibitem{DozierSilver2} 
R.B. Dozier and J.W. Silverstein.
\newblock Analysis of the limiting spectral distribution of large dimensional information-plus-noise type matrices.
\newblock {\em J. Multivariate. Anal.}, vol. 98, no. 6: 1099--1122 , 2007.

\bibitem{DHLN}
 J. Dumont, W. Hachem, S. Lasaulce, Ph. Loubaton and J. Najim.
 On the Capacity Achieving Covariance Matrix for Rician MIMO Channels: An Asymptotic Approach {\em  IEEE Transactions on Information Theory} Vol. 56, n° 3, pp. 1048-1069, March 2010. 



\bibitem{GuZe03}
A.~Guionnet and B.~Zegarlinski.
\newblock Lectures on Logarithmic Sobolev inequalities.
\newblock In {\em S\'eminaire de {P}robabilit\'es, {XXXVI}}, volume 1801 of
  {\em Lecture Notes in Math.}. Springer, Berlin, 2003.

\bibitem{HLN}
W. Hachem, P. Loubaton and J. Najim.
Deterministic Equivalents for certain functionals of large random matrices.
{\em Ann. Appl. Probab.} (17), no. 3, 875--930, 2007.











\bibitem{HT}
U.~Haagerup and S.~Thorbj{\o}rnsen.
Random matrices with complex Gaussian entries
{\em Expo. 
Math. }
21 
(2003): 
293-337 


\bibitem{HaaThor05}
U.~Haagerup and S.~Thorbj{\o}rnsen.
\newblock A new application of random matrices: {${\rm Ext}(C^*_{\rm
  red}(F_2))$} is not a group.
\newblock {\em Ann. of Math. (2)}, 162(2):711--775, 2005.






\bibitem{KKP96}
A.~M. Khorunzhy, B.~A. Khoruzhenko, and L.~A. Pastur.
\newblock Asymptotic properties of large random matrices with independent
  entries.
\newblock {\em J. Math. Phys.}, 37(10):5033--5060, 1996.


\bibitem{PL}
  O.~Ledoit and S.~P{\'e}ch{\'e}
\newblock Eigenvectors of some large sample covariance matrix ensembles.
\newblock {\em Probab. Theory Relat. Fields}, online 2010.



\bibitem{Ledoux01} M.~Ledoux.
\newblock {\em The concentration of Measure Phenomenon}.
\newblock American Mathematical Society, Providence, RI, 2001.

\bibitem{LV}  P. Loubaton and P. Vallet. Almost sure localization of the eigenvalues in a Gaussian  information-plus-noise model. Application to the spiked models 
{\em Electronic Journal of Probability}, vol. 16 : 1934-1959, 2011. 

\bibitem{M} Hans Maassen. 
 Addition of freely independent random variables.
{\em J. Funct. Anal.}
 106(2):409-438, 1992.


\bibitem{PS} L.A. Pastur and M. Shcherbina
\newblock {\em Eigenvalue Distribution of Large Random Matrices}.
Mathematical surveys and monographs.
\newblock American Mathematical Society, 2011.




\bibitem{Paul}
D. Paul.
\newblock Asymptotics of sample eigenstructure for a large dimensional spiked covariance model
\newblock {\em Statist. Sinica}, 17 (4):1617--1642, 2007.






\bibitem{VLM} P. Vallet, P. Loubaton and X. Mestre. Improved Subspace Estimation for Multivariate Observations of High Dimension: The Deterministic Signal Case.
{\em IEEE Transactions on Information Theory}, vol. 58, no. 2, 2012. 


\bibitem{VDN} D.V. Voiculescu, K. Dykema, and A. Nica, Free random variables, CRM Monograph Series,
vol. 1, American Mathematical Society, Providence, RI, 1992, ISBN 0-8218-6999-X, A
noncommutative probability approach to free products with applications to random matrices,
operator algebras and harmonic analysis on free groups.

\bibitem{wangwang}   F.-Y. Wang and J. Wang. 
Functional inequalities for convolution probability measures. 
{\em  Ann. Inst. Henri Poincaré Probab. Stat.} 52, no. 2: 898--914, 2016. 

\bibitem{Xi} J.-s. Xie. The convergence on spectrum of sample covariance 
matrices for
information-plus-noise type data.
{\em Appl. Math. J. Chinese Univ.} Ser. B, 27(2):181–191, 2012.



\end{thebibliography}
\end{document}